\documentclass[11pt]{amsart}
\usepackage{amsthm,amssymb,amsmath,soul,bropd,float,setspace,enumerate,ragged2e,setspace}
\usepackage{color,verbatim}
\usepackage{hyperref}
\usepackage{tikz}
\usepackage[top=1in, left=0.8in, right=0.8in, bottom=1in, marginpar=2cm]{geometry}

\def\XXint#1#2#3{{\setbox0=\hbox{$#1{#2#3}{\int}$ }
		\vcenter{\hbox{$#2#3$ }}\kern-.5\wd0}}

\newtheorem{theorem}{Theorem}[section]
\newtheorem{lemma}[theorem]{Lemma}
\newtheorem{proposition}[theorem]{Proposition}
\newtheorem{corollary}[theorem]{Corollary}

\newtheorem{remark}[theorem]{Remark}

\newtheorem{assumption}[theorem]{Assumption}
\newtheorem{definition}{Definition}[section]

\def\d{\delta}

\def\F{\Phi}

\def\0{\varnothing}
\def\1{\left(}
\def\2{\right)}
\def\3{\left[}
\def\4{\right]}
\def\5{\left\{}
\def\6{\right\}}
\def\8{\infty}

\def\gg{\geqslant}

\begin{document}
\title[Non-existence of horizontally flat singularity]{The non-existence of horizontally flat singularity for steady axisymmetric free surface flows near stagnation points}

\author{Lili Du}
\address[]{School of Mathematical Sciences, Shenzhen University, Shenzhen, P. R. China}
\email{dulili@szu.edu.cn}
\thanks{The first author is supported by National Nature Science Foundation of China Grant 11971331, 12125102, and Sichuan Youth Science and Technology Foundation 2021JDTD0024.}

\author{Chunlei Yang}
\address[]{Department of Mathematics, Sichuan University, Chengdu, P. R. China}
\email{yang\_chunlei@stu.scu.edu.cn}
\begin{abstract}
	In a recent research on degenerate points of steady axisymmetric gravity flows with general vorticity, it has been shown that the possible asymptotics near any stagnation point must be the “Stokes corner”, the “horizontal cusp”, or the “horizontal flatness”(Theorem 1.1, Du, Huang, Pu, \emph{Commun. Math. Phys.}, 400, 2137-2179, 2023). In this paper, we focus on the horizontally flat singularity and show that it is \emph{not possible}, and therefore the “Stokes corner” and the “cusp” are the only possible asymptotics at the stagnation points. The basic idea of our proof relies on a perturbation of the \emph{frequency formula} for the two-dimensional problem (V\v{a}rv\v{a}ruc\v{a}, Weiss, \emph{Acta Math.}, 206, 363-403, 2011). Our analysis also suggests that, for steady axisymmetric rotational gravity flows, the singular asymptotic profiles at stagnation points are similar to the scenario observed in two-dimensional waves with vorticity. (V\v{a}rv\v{a}ruc\v{a}, Weiss, \emph{Ann. I. H. Poincar\'{e}-AN}, 29, 861-885, 2012).
\end{abstract}

\maketitle
\begin{spacing}{0.01}
	\tableofcontents
\end{spacing}
\section{Introduction}
The model problem we have in mind is the following semilinear Bernoulli-type free boundary problem, which describes an incompressible axisymmetric rotational flow acted on by gravity and with a free surface,
\begin{equation}\label{Formula: model problem}
	\begin{cases}
		\displaystyle\operatorname{div}\br{\frac{1}{x}\nabla\psi} = -xf(\psi) & \text{ in } \Omega\cap\{\psi > 0\}, \\
		\displaystyle\frac{1}{x^{2}}|\nabla\psi|^{2} = -y & \text{ on } \Omega\cap\partial \{\psi > 0\}.
	\end{cases}
\end{equation}
Here $\psi$ is the Stokes stream function and $\Omega$ is a connected open subset relative to the right half-plane $\mathbb{R}_{+}^{2}=\{(x,y)\in\mathbb{R}^{2}:x\geqslant0\}$ (This implies that $(\Omega\cap\{x=0\})\cap\partial\Omega=\varnothing$). The given function $f$ in the first equation of \eqref{Formula: model problem} represents the strength of the vorticity, and $\Omega\cap\partial\{\psi>0\}$ denotes the free surface of the flow. Physically the equation \eqref{Formula: model problem} arises for example as the “axially symmetric rising jet” or “a bubble within a semi-infinitely long cylindrical tube” and we kindly refer readers to \cite[Section 2]{DHP2022} for a comprehensive description.

Regarding \eqref{Formula: model problem} as a free boundary problem, our primary interest lies in the geometry profiles of the free surface at which the relative fluid velocity $(\frac{1}{x}\psi_{y},-\frac{1}{x}\psi_{x})$ is the zero vector, and the regularity of the free surface when it is not the zero vector. In the formal case, a glance of the second equation in \eqref{Formula: model problem} indicates that such points are distributed along the axis of symmetry $\{x=0\}$ or the set $\{y=0\}$. As for the latter case, it is evident that at those free boundary points that are away from both the axis of symmetry and the set $\{y=0\}$, the gradient of $\psi$ can never reduce to zero. Due to the degeneracy or the non-degeneracy of the free boundary condition, we classify them as the set of stagnation points $S_{\psi}^{s}$ (Type I points, which includes those free boundary points $(x_{0},0)$, $x_{0}>0$), the set of symmetric axis points $S_{\psi}^{a}$ (Type II points, which includes free boundary points $(0,y_{0})$, $y_{0}<0$), the origin $O=(0,0)$ (which can be regarded as the intersection of Type I and Type II points), and the set of non-degenerate points $N_{\psi}$ (which consists of points on the free boundary points $(x_{0},y_{0})$, $x_{0}>0$, $y_{0}<0$).
We now briefly review some results known to us with respect to the singularity and regularity of the free surface near the degenerate points (Type I, Type II and the origin $O$) and non-degenerate points ($N_{\psi}$). In \cite{DHP2022}, the first author and et al. demonstrated that the possible singular asymptotics for any point in $S_{\psi}^{s}$ are the “Stokes corner” wherein the free surface exhibits a corner of $\frac{2}{3}\pi$, a “horizontal cusp”, or a “horizontal flatness” (Figure \ref{Fig: asymptotics for typeI point}).
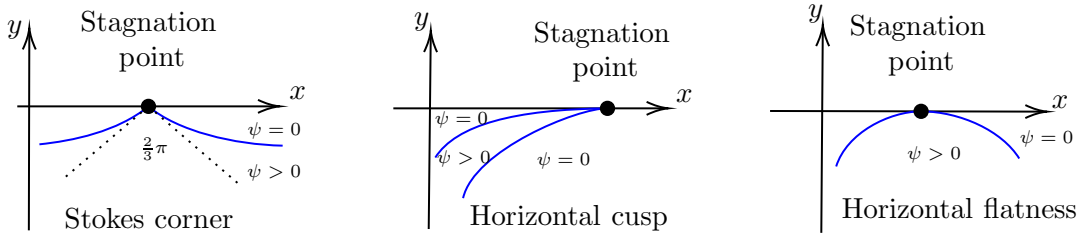
\begin{figure}[!ht]
	\centering
	\tikzset{every picture/.style={line width=0.75pt}} 
	
	\begin{tikzpicture}[x=0.75pt,y=0.75pt,yscale=-1,xscale=1]
		
		\draw    (20.5,189.25) -- (20.74,89) ;
		\draw [shift={(20.74,87)}, rotate = 90.14] [color={rgb, 255:red, 0; green, 0; blue, 0 }  ][line width=0.75]    (10.93,-3.29) .. controls (6.95,-1.4) and (3.31,-0.3) .. (0,0) .. controls (3.31,0.3) and (6.95,1.4) .. (10.93,3.29)   ;
		\draw    (14.67,126.67) -- (145,126.67) ;
		\draw [shift={(147,126.67)}, rotate = 180] [color={rgb, 255:red, 0; green, 0; blue, 0 }  ][line width=0.75]    (10.93,-3.29) .. controls (6.95,-1.4) and (3.31,-0.3) .. (0,0) .. controls (3.31,0.3) and (6.95,1.4) .. (10.93,3.29)   ;
		\draw    (222.17,188.58) -- (223.2,90.33) ;
		\draw [shift={(223.22,88.33)}, rotate = 90.6] [color={rgb, 255:red, 0; green, 0; blue, 0 }  ][line width=0.75]    (10.93,-3.29) .. controls (6.95,-1.4) and (3.31,-0.3) .. (0,0) .. controls (3.31,0.3) and (6.95,1.4) .. (10.93,3.29)   ;
		\draw    (204.33,127.67) -- (341,127.67) ;
		\draw [shift={(343,127.67)}, rotate = 180] [color={rgb, 255:red, 0; green, 0; blue, 0 }  ][line width=0.75]    (10.93,-3.29) .. controls (6.95,-1.4) and (3.31,-0.3) .. (0,0) .. controls (3.31,0.3) and (6.95,1.4) .. (10.93,3.29)   ;
		\draw [color={rgb, 255:red, 0; green, 0; blue, 255 }  ,draw opacity=1 ]   (239.5,173) .. controls (245,148.8) and (297.17,127.5) .. (312.33,127.67) ;
		\draw [color={rgb, 255:red, 0; green, 0; blue, 255 }  ,draw opacity=1 ]   (225.5,152.5) .. controls (236,136) and (265.5,128.17) .. (312.33,127.67) ;
		\draw    (420.17,190.5) -- (420.92,87.33) ;
		\draw [shift={(420.94,85.33)}, rotate = 90.42] [color={rgb, 255:red, 0; green, 0; blue, 0 }  ][line width=0.75]    (10.93,-3.29) .. controls (6.95,-1.4) and (3.31,-0.3) .. (0,0) .. controls (3.31,0.3) and (6.95,1.4) .. (10.93,3.29)   ;
		\draw    (394.33,129) -- (532,129.33) ;
		\draw [shift={(534,129.33)}, rotate = 180.14] [color={rgb, 255:red, 0; green, 0; blue, 0 }  ][line width=0.75]    (10.93,-3.29) .. controls (6.95,-1.4) and (3.31,-0.3) .. (0,0) .. controls (3.31,0.3) and (6.95,1.4) .. (10.93,3.29)   ;
		\draw [color={rgb, 255:red, 0; green, 0; blue, 255 }  ,draw opacity=1 ]   (25.83,145.83) .. controls (42.58,144.08) and (65.17,139.5) .. (80.83,126.67) ;
		\draw [color={rgb, 255:red, 0; green, 0; blue, 255 }  ,draw opacity=1 ]   (80.83,126.67) .. controls (95.5,139.17) and (124,146) .. (148.5,146.5) ;
		\draw  [dash pattern={on 0.84pt off 2.51pt}]  (80.83,126.67) -- (127.5,166.83) ;
		\draw  [dash pattern={on 0.84pt off 2.51pt}]  (80.83,126.67) -- (37.17,163.83) ;
		\draw [color={rgb, 255:red, 0; green, 0; blue, 255 }  ,draw opacity=1 ]   (427.67,157) .. controls (433,140) and (454.33,129.33) .. (470.5,129.17) ;
		\draw [color={rgb, 255:red, 0; green, 0; blue, 255 }  ,draw opacity=1 ]   (470.5,129.17) .. controls (490.33,129.33) and (511.33,139) .. (520,153) ;
		\draw    (80.83,126.67) ;
		\draw [shift={(80.83,126.67)}, rotate = 0] [color={rgb, 255:red, 0; green, 0; blue, 0 }  ][fill={rgb, 255:red, 0; green, 0; blue, 0 }  ][line width=0.75]      (0, 0) circle [x radius= 3.35, y radius= 3.35]   ;
		\draw    (312.33,127.67) ;
		\draw [shift={(312.33,127.67)}, rotate = 0] [color={rgb, 255:red, 0; green, 0; blue, 0 }  ][fill={rgb, 255:red, 0; green, 0; blue, 0 }  ][line width=0.75]      (0, 0) circle [x radius= 3.35, y radius= 3.35]   ;
		\draw    (470.5,129.17) ;
		\draw [shift={(470.5,129.17)}, rotate = 0] [color={rgb, 255:red, 0; green, 0; blue, 0 }  ][fill={rgb, 255:red, 0; green, 0; blue, 0 }  ][line width=0.75]      (0, 0) circle [x radius= 3.35, y radius= 3.35]   ;
		
		\draw (75,141.07) node [anchor=north west][inner sep=0.75pt]  [font=\tiny]  {$\frac{2}{3} \pi $};
		\draw (129.17,154.4) node [anchor=north west][inner sep=0.75pt]  [font=\tiny]  {$\psi  >0$};
		\draw (129.23,133.13) node [anchor=north west][inner sep=0.75pt]  [font=\tiny]  {$\psi =0$};
		\draw (37,176.67) node [anchor=north west][inner sep=0.75pt]   [align=left] {Stokes corner};
		\draw (150.5,115.4) node [anchor=north west][inner sep=0.75pt]    {$x$};
		\draw (8,81.9) node [anchor=north west][inner sep=0.75pt]    {$y$};
		\draw (535,119.9) node [anchor=north west][inner sep=0.75pt]    {$x$};
		\draw (412,73.9) node [anchor=north west][inner sep=0.75pt]    {$y$};
		\draw (429,171.5) node [anchor=north west][inner sep=0.75pt]   [align=left] {Horizontal flatness};
		\draw (462,145.4) node [anchor=north west][inner sep=0.75pt]  [font=\tiny]  {$\psi  >0$};
		\draw (518.77,137.07) node [anchor=north west][inner sep=0.75pt]  [font=\tiny]  {$\psi =0$};
		\draw (241.33,175) node [anchor=north west][inner sep=0.75pt]   [align=left] {Horizontal cusp};
		\draw (345.5,116.9) node [anchor=north west][inner sep=0.75pt]    {$x$};
		\draw (213.33,77.9) node [anchor=north west][inner sep=0.75pt]    {$y$};
		\draw (275.17,148.07) node [anchor=north west][inner sep=0.75pt]  [font=\tiny]  {$\psi =0$};
		\draw (223.67,127.4) node [anchor=north west][inner sep=0.75pt]  [font=\tiny]  {$\psi =0$};
		\draw (224.7,147.23) node [anchor=north west][inner sep=0.75pt]  [font=\tiny]  {$\psi  >0$};
		\draw (45,77.33) node [anchor=north west][inner sep=0.75pt]   [align=left] {Stagnation \\ \ \ \ \ point};
		\draw (273.67,82.67) node [anchor=north west][inner sep=0.75pt]   [align=left] {Stagnation \\ \ \ \ \ point};
		\draw (433,79.33) node [anchor=north west][inner sep=0.75pt]   [align=left] {Stagnation \\ \ \ \ \ point};
	\end{tikzpicture}
	\caption{Asymptotics for Type I point.}
	\label{Fig: asymptotics for typeI point}
\end{figure}

Furthermore, at any point in axis $S_{\psi}^{a}$, the singular profiles of the free surface must be a “cusp”. At the origin $O$, they concluded that the possibilities are the “Garabedian pointed bubble”, a “horizontal cusp”, or a “horizontal flatness” (we refer readers to \cite[Table 1]{DHP2022} for these singular asymptotics). These findings suggest that the asymptotic profiles in proximity to various degenerate points for axisymmetric inviscid gravity flows with vorticity are akin to those observed in the irrotational case \cite{VW2014}. Ultimately, it is not beyond our expectations that the free boundary near all non-degenerate points can be expressed as a graph of a  $C^{1,\alpha}$ $(0<\alpha<1)$ smooth function, and we would like to refer readers to our recent work \cite{DY2023} on this subject.

In this paper, we focus on \emph{horizontally flat singularities} near stagnation points (The third graph in Figure \ref{Fig: asymptotics for typeI point}). Originally, they were observed and excluded during the investigation of Stokes conjecture for water waves with zero vorticity \cite[Theorem B]{VW2011}, and also for water waves with non-zero vorticity \cite[Theorem 4.6]{VW2012}. Furthermore, horizontally flat singularities do not exist at stagnation points of axisymmetric gravity irrotational flows \cite[Theorem 3.8]{VW2014}. Physical intuitions that relate to this kind of singularities can be interpreted as there being an infinite number of connected components of the air region. Mathematically, these singularities correspond to a specific density, which we shall discuss in the subsequent section (cf. Proposition \ref{Proposition: density estimates}), at the stagnation points. This density, which arises as the limits of rescaled variational solutions of the Weiss-boundary adjusted energy, plays an important role in classifying asymptotics at the stagnation points. One of the primary contributions of the work  \cite[Theorem 1.1]{DHP2022} was to compute all the possible blow-up limits and densities at the stagnation points when the vorticity is considered, and the objective of this article is to prove that the set of stagnation point that corresponds to the density of horizontal flat singularities is an empty set.

The appearance of the horizontally flat singularities and the cusp singularities is caused by the lack of compactness for variational solutions of the problem \eqref{Formula: model problem}. It should be worth noting that the cusp singularities near the stagnation points were proved to be not possible when assuming the strong Bernstein estimates for water waves without vorticity \cite[Lemma 4.4]{VW2011}, and were conjectured to be impossible for water waves with vorticity with the additional Rayleigh-Taylor condition \cite[Remark 4.7]{VW2012}. As for the horizontally flat singularities, V\v{a}rv\v{a}ruc\v{a} and Weiss considered for the first time the so-called \emph{frequency formula} \cite[Section 7]{VW2011}, and successfully applied this formula to exclude them. The frequency of the harmonic functions was initially observed by Almgren in \cite{A2000} for $Q$-valued harmonic functions and later developed by Garofalo and Lin to more general elliptic equations in \cite{GL1986,GL1987}. The aim of this paper is to develop ingenious tools introduced in \cite{VW2012,VW2014}: a frequency formula (presented in Theorem \ref{Theorem: freq}) in the context of axisymmetric rotational waves. The new frequency formula helps us to investigate the blow-up limits of the frequency function (presented in Proposition \ref{Proposition: vm}). The situation is complicated here by the fact that the axisymmetric problem is \emph{not} as the same as the true two-dimensional problem, and it is unavoidable to deal with the integrability of those additional terms caused by the axis of symmetry. A major part of this work, which is originally inspired by \cite[Remark 3.11]{VW2014}, is to prove that these terms are indeed behave like small \emph{perturbations} of the two-dimensional problem (Presented in Proposition \ref{Lemma: V(r)}).

Our method in the present work can also be used to study the degenerate points at the origin (cf. Proposition 5.3 in \cite{DHP2022}). We can find a non-trivial solution by following the same steps as in Section \ref{Section: frequency formula}-Section \ref{Section: conclusions}. This solution is a homogeneous function of degree $3$ and has the form 
\begin{align*}
	u_{0}(x,y)=\frac{x_{1}^{2}x_{2}^{+}}{\sqrt{\int_{\partial B_{1}^{+}}x_{1}^{3}(x_{2}^{+})^{2}d\mathcal{H}^{1}}}.
\end{align*}
Notice that it was proved in \cite[Proposition 5.3]{DHP2022} that the only non-trivial asymptotics suggested by the invariant scaling of the equation at the origin is the Garabedian pointed bubble with water above air. Thus, our approach gives the possibility (existence) of a solution that has a higher growth than that asked by the invariant scaling.

The flow of the paper is the following: after clarifying some notations and some preliminary results in Section \ref{Section: notation}, we follow in Section \ref{Section: frequency formula} the approach in \cite{VW2012} in order to compute the frequency formula in our settings. Then, we apply in Section \ref{Section: blow-up limits} the ideas of \cite{VW2012,VW2014} to study blow-up limits of the frequency sequence. Finally, we prove that the set of horizontally flat singularities is indeed empty and we prove this result in Section \ref{Section: conclusions}.
\section{Notations and technical tools}\label{Section: notation}
For the abuse of confusion, we adhere some notations utilized in the previous works. In what follows, we denote by $\chi_{A}$ the characteristic function of a set $A$, by $X=(x,y)$ a point in  $\mathbb{R}^{2}$, by $B_{r}(X_{0}):=\{X\in\mathbb{R}^{2}:|X-X_{0}|<r\}$ the ball of center $X_{0}$ and radius $r$ and by $B_{r}$ the ball $B_{r}(0)$. Also, $\mathcal{H}^{s}$ shall denote the $s$-dimensional Hausdorff measure and $\nu$ will always refer to the outer normal on a given surface. We will say $O(1)$ and $O(r)$, etc., if  $|O(1)|\leqslant C$, $|O(r)|\leqslant Cr$ for some positive constant $C$.

Moreover, $dX:=dxdy$ shall stand for the volume element and we further introduce weighted Lebesgue $L_{\mathrm{w}}^{2}(E)$ and Sobolev space $W_{\mathrm{w}}^{1,2}(E)$, which are defined for any open subset $E$ of $\mathbb{R}_{+}^{2}$ by 
\begin{align*}
	L_{\mathrm{w}}^{2}(E):=\left\lbrace\psi\colon E\to\mathbb{R}\text{ is measurable and }\int_{E}\frac{1}{x}|\psi|^{2}\:dX<+\infty\right\rbrace,
\end{align*}
and
\begin{align*}
	W_{\mathrm{w}}^{1,2}(E):=\left\lbrace\psi\in L_{\mathrm{w}}^{2}(E):\pd{\psi}{x}\in L_{\mathrm{w}}^{2}(E)\text{ and }\pd{\psi}{y}\in L_{\mathrm{w}}^{2}(E)\right\rbrace,
\end{align*}
respectively. For the sake of convenience we are going to reflect the problem at the line $\{y=0\}$, namely, we study the solutions of the problem 
\begin{equation}\label{Formula: model problem1}
	\begin{cases}
		\displaystyle\operatorname{div}\br{\frac{1}{x}\nabla\psi} = -xf(\psi) & \text{ in } \Omega\cap \{\psi > 0\}, \\
		\displaystyle\frac{1}{x^{2}}|\nabla\psi|^{2} = y & \text{ on } \Omega\cap\partial\{\psi > 0\}.
	\end{cases}
\end{equation}
The nonlinearity $f$ in the first equation of \eqref{Formula: model problem1} is assumed at present a continuous function with primitive $F(z):=\int_{0}^{z}f(s)ds$. We define the set of stagnation points $S_{\psi}^{s}$ by 
\begin{align*}
	S_{\psi}^{s}:=\{X_{0}=(x_{0},0)\in\Omega\cap\partial\{\psi>0\}\colon x_{0}>0\}. 
\end{align*}
Since our problem is completely local, we do not impose any boundary condition on $\partial\Omega$. Finally, for any $X_{0}\in S_{\psi}^{s}$, we define the number
\begin{align}\label{Formula: d}
	\delta:=\frac{\min\{x_{0},\operatorname{dist}(X_{0},\partial\Omega)\}}{2}.
\end{align}
We now introduce our notion of a \emph{variational solution} of the problem \eqref{Formula: model problem1}.
\begin{definition}[Variational solutions]
	The function $\psi\in W_{\mathrm{w},\mathrm{loc}}^{1,2}(\Omega)$ such that $\psi\geqslant0$ in $\Omega$, $\psi=0$ in $\Omega\cap\{x=0\}$ is called a \emph{variational solution} of the problem \eqref{Formula: model problem1}, provided that $\psi\in C^{0}(\Omega)\cap C^{2}(\Omega\cap\{ \psi>0 \})$, $\psi\equiv0$ in $\{y\leqslant0\}$,
	\begin{align*}
		\lim_{\substack{X\to X_{0},\\ X\in\Omega\cap\{\psi>0\}}}\frac{1}{x}\pd{\psi}{y}=0,\qquad\lim_{\substack{X\to X_{0},\\ X\in\Omega\cap\{\psi>0\}}}\frac{1}{x}\pd{\psi}{x}\text{ exists, }
	\end{align*}
	for any $X_{0}\in\Omega\cap\{x=0\}$, and
	\begin{align}\label{Formula: Variational solution}
		\begin{split}
			&\int_{\Omega}\left(\frac{1}{x}|\nabla\psi|^{2}-2xF(\psi)+xy\chi_{\{\psi>0\}}\right)\operatorname{div}\phi-\frac{2}{x}\nabla\psi D\phi\nabla\psi dX\\
			&\qquad\qquad+\int_{\Omega}\left(-\frac{1}{x^{2}}|\nabla\psi|^{2}-2F(\psi)+y\chi_{\{\psi>0\}}\right)\phi_{1}dX+\int_{\Omega}x\chi_{\{\psi>0\}}\phi_{2}dX=0,
		\end{split}
	\end{align}
	for any $\phi=(\phi_{1},\phi_{2})\in C_{0}^{\infty}(\Omega;\mathbb{R}^{2})$ such that $\phi_{1}=0$ on $\{x=0\}$.
\end{definition}
\begin{remark}
	The equation \eqref{Formula: Variational solution} is nothing but the Euler-Lagrange equation of the energy 
	\begin{align*}
		E(\phi):=\int_{\Omega}\left(\frac{1}{x}|\nabla\phi|^{2}-2xF(\phi)+xy\chi_{\{\phi>0\}}\right)dX
	\end{align*}
	with respect to domain variation. Assume that the free boundary $\varGamma:=\Omega\cap\partial\{\psi>0\}$ and $\psi$ are $C^{2}$ smooth, then an integration by parts both in the water phase and on the free surface of the formula \eqref{Formula: Variational solution} shows that $\psi$ solves the equation \eqref{Formula: model problem1} in the classical sense.
\end{remark}
\begin{remark}\label{Remark: Energy identity}
	In this remark, we introduce an energy identity (cf. \cite[(3.3)]{DHP2022}) which is helpful for our future discussion. For any $X_{0}\in S_{\psi}^{s}$ and any $r\in(0,\delta)$, the following identity holds.
	\begin{align}\label{Formula: Energy identity}
		\int_{B_{r}(X_{0})}\left(\frac{1}{x}|\nabla\psi|^{2}-x\psi f(\psi)\right)dX=\int_{\partial B_{r}(X_{0})}\frac{1}{x}\psi\nabla\psi\cdot\nu d\mathcal{H}^{1}.
	\end{align}
\end{remark}
The technical tools at our disposition comprise a monotonicity formula and a density estimate introduced in \cite{DHP2022}. For the sake of completeness, let us state them respectively.
\begin{proposition}[Weiss's monotonicity formula, {\cite[Lemma 3.1]{DHP2022}}]
	Let $\psi$ be a variational solution of the problem \eqref{Formula: model problem1} and let $X_{0}\in S_{\psi}^{s}$. For any $r\in(0,\delta)$, define
	\begin{subequations}
		\begin{align}
			&I_{X_{0},\psi}(r)=I(r)=r^{-3}\int_{B_{r}(X_{0})}\left(\frac{1}{x}|\nabla\psi|^{2}-x\psi f(\psi)+xy\chi_{\{\psi>0\}}\right)dX, \label{Formula: I(r)}\\
			&J_{X_{0},\psi}(r)=J(r)=r^{-4}\int_{\partial B_{r}(X_{0})}\frac{1}{x}\psi^{2}d\mathcal{H}^{1}, \label{Formula: J(r)}\\
			&K_{X_{0},\psi}(r)=K(r)=r\int_{\partial B_{r}(X_{0})}(2xF(\psi)-x\psi f(\psi))d\mathcal{H}^{1}+\int_{B_{r}(X_{0})}(2x_{0}F(\psi)-6xF(\psi))dX,\label{Formula: K(r)}\\
			&I_{1,X_{0},\psi}(r)=I_{1}(r)=\int_{B_{r}(X_{0})}-\frac{x-x_{0}}{x^{2}}|\nabla\psi|^{2}dX,\label{Formula: I1(r)}\\
			&I_{2,X_{0},\psi}(r)=I_{2}(r)=\int_{B_{r}(X_{0})}(x-x_{0})y\chi_{\{\psi>0\}}dX,\label{Formula: I2(r)},
		\end{align}
		and
		\begin{flalign}\label{Formula: J1(r)}
			\qquad&J_{1,X_{0},\psi}(r)=J_{1}(r)=\int_{\partial B_{r}(X_{0})}\frac{x-x_{0}}{x^{2}}\psi^{2}d\mathcal{H}^{1}.&
		\end{flalign}
		Then the Weiss-boundary adjusted energy
		\begin{align}\label{Formula: Weiss-boundary adjusted energy}
			\Phi_{X_{0},\psi}(r)=\Phi(r)=I(r)-\frac{3}{2}J(r)
		\end{align}
		satisfies for a.e. $r\in(0,\delta)$,
		\begin{align}\label{Formula: derivatives of Weiss-boundary adjusted energy}
			\Phi'(r)=2r^{-3}\int\limits_{\partial B_{r}(X_{0})}\frac{1}{x}\left(\nabla\psi\cdot\nu-\frac{3}{2}\frac{\psi}{r}\right)^{2}d\mathcal{H}^{1}+r^{-4}K(r)+r^{-4}\sum_{i=1}^{2}I_{i}(r)+\frac{3}{2}r^{-5}J_{1}(r).
		\end{align}
	\end{subequations}
\end{proposition}
\begin{proof}
	After choosing particular test function into \eqref{Formula: Variational solution} and a straightforward calculation (cf. \cite[proof of Lemma 3.1]{DHP2022}), we obtain
	\begin{align}\label{Formula: I'(r)}
		I'(r)=r^{-4}\left(2r\int_{\partial B_{r}(X_{0})}\frac{1}{x}(\nabla\psi\cdot\nu)^{2}d\mathcal{H}^{1}-3\int_{\partial B_{r}(X_{0})}\frac{1}{x}\psi\nabla\psi\cdot\nu d\mathcal{H}^{1}+K(r)\right)+r^{-4}\sum_{i=1}^{2}I_{i}(r),
	\end{align}
	and
	\begin{align}\label{Formula: J'(r)}
		J'(r)=r^{-5}\left(2r\int_{\partial B_{r}(X_{0})}\frac{1}{x}\psi\nabla\psi\cdot\nu d\mathcal{H}^{1}-3\int_{\partial B_{r}(X_{0})}\frac{1}{x}\psi^{2}d\mathcal{H}^{1}-J_{1}(r)\right).
	\end{align}
\end{proof}
We now recall the definition of weak solutions of the problem \eqref{Formula: model problem1}.
\begin{definition}[Weak solutions]
	The function $\psi\in W_{\mathrm{w}}^{1,2}(\Omega)$ is called a \emph{weak solution} of the problem \eqref{Formula: model problem1}, provided that $\psi$ is a variational solution of the problem \eqref{Formula: model problem1} and the topological free boundary $\partial\{\psi>0\}\cap\Omega\cap\{x>0\}\cap\{y\neq0\}$ is locally a $C^{2,\alpha}$ curve.
\end{definition}
\begin{remark}
	The smoothness of the topological free boundary of the problem \eqref{Formula: model problem1} was recently proved in \cite[Theorem 1.1]{DY2023}.
\end{remark}
For weak solutions of the problem \eqref{Formula: model problem1}, the first author and his collaborators prove the following density estimates with respect to the stagnation points (cf. \cite[Proposition 3.3]{DHP2022}).
\begin{proposition}[Density estimates, {\cite[Proposition 3.3]{DHP2022}}]\label{Proposition: density estimates}
	Let $\psi$ be a weak solution with the following growth assumption
	\begin{align}\label{Formula: growth assumption}
		\frac{|\nabla\psi|^{2}}{x^{2}}\leqslant C(|y|+|x-x_{0}|)\quad\text{ in }\quad B_{r}(X_{0}),\quad\text{ for any }\quad X_{0}\in\Omega\cap\partial\{\psi>0\},
	\end{align}
	and let $X_{0}\in S_{\psi}^{s}$. Then the following statements hold:
	\begin{enumerate}
		\item $\lim_{r\to0+}\Phi(r)$ exists and
		\begin{align*}
			\Phi(0+)\in\left\lbrace 0,x_{0}\int\limits_{B_{1}}y^{+}\chi_{ \{x\colon\pi/6<\theta<5\pi/6\} }dX,x_{0}\int\limits_{B_{1}}y^{+}dX\right\rbrace.
		\end{align*}
		\item In the case $\Phi(0+)=x_{0}\int_{B_{1}}y^{+}\chi_{ \{x\colon\pi/6<\theta<5\pi/6\} }dX$, we have
		\begin{align*}
			\frac{\psi((x_{0},0)+r(x,y))}{r^{3/2}}\to\frac{\sqrt{2}x_{0}}{3}\rho^{3/2}\cos\left(\frac{3}{2}\left(\min\left(\max\left(\theta,\frac{\pi}{6}\right),\frac{5\pi}{6}\right)-\frac{\pi}{2}\right)\right)\quad\text{ as }\quad r\to0+,
		\end{align*}
		strongly in $W_{\mathrm{loc}}^{1,2}(\mathbb{R}^{2})$ and locally uniformly in $\mathbb{R}^{2}$, where $X=(\rho\cos\theta,\rho\sin\theta)$.
		\item In the case $\Phi(0+)\in\{0,x_{0}\int_{B_{1}}y^{+}dX\}$, we have
		\begin{align*}
			\frac{\psi((x_{0},0)+r(x,y))}{r^{3/2}}\to0\quad\text{ as }\quad r\to0+,
		\end{align*}
		strongly in $W_{\mathrm{loc}}^{1,2}(\mathbb{R}^{2})$ and locally uniformly in $\mathbb{R}^{2}$.
	\end{enumerate}
\end{proposition}
In order to better describe the asymptotics of different singularities, authors in \cite{DHP2022} imposed an assumption (cf. \cite[Assumption 1.2]{DHP2022}) on the free boundary of the problem \eqref{Formula: model problem1}, which we states here.
\begin{assumption}[Curve assumption]\label{Assumption: curve assumption}
	Suppose that in a small neighborhood $B_{r_{0}}(X_{0})$ of $X_{0}\in S_{\psi}^{s}$, the free boundary $\partial\{\psi>0\}\cap B_{r_{0}}(X_{0})$ is a \emph{continuous injective} curve $\gamma(t):=(\gamma_{1}(t),\gamma_{2}(t)):I\to\mathbb{R}^{2}$, where $I$ is an interval of $\mathbb{R}$ containing $0$ and $\gamma(0)=X_{0}$.
\end{assumption}
\begin{remark}
	It should be worth noted that the asymptotic profiles related to different densities can be graphed explicitly when assuming the Assumption \ref{Assumption: curve assumption}. In particular, we obtain in case (3) of Proposition \ref{Proposition: density estimates} that if $\Phi(0+)=x_{0}\int_{B_{1}}y^{+}dX$, then (cf. Figure \ref{Fig: horizontal flatness asymptotics}) $\gamma_{1}(t)\neq x_{0}$ in $(-t_{1},t_{1})\setminus\{0\}$, $\gamma_{1}-x_{0}$ changes sign at $t=0$, and
	\begin{align*}
		\lim_{t\to0}\frac{\gamma_{2}(t)}{\gamma_{1}(t)-x_{0}}=0.
	\end{align*}
	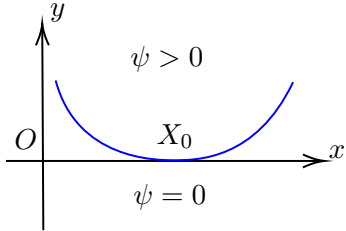
\begin{figure}[!ht]
		\tikzset{every picture/.style={line width=0.75pt}} 
		\begin{tikzpicture}[x=0.75pt,y=0.75pt,yscale=-1,xscale=1]
			
			\draw    (204.33,205) -- (203.8,102.72) ;
			\draw [shift={(203.79,100.72)}, rotate = 89.7] [color={rgb, 255:red, 0; green, 0; blue, 0 }  ][line width=0.75]    (10.93,-3.29) .. controls (6.95,-1.4) and (3.31,-0.3) .. (0,0) .. controls (3.31,0.3) and (6.95,1.4) .. (10.93,3.29)   ;
			\draw    (185.5,170.05) -- (344.19,170.05) ;
			\draw [shift={(346.19,170.05)}, rotate = 180] [color={rgb, 255:red, 0; green, 0; blue, 0 }  ][line width=0.75]    (10.93,-3.29) .. controls (6.95,-1.4) and (3.31,-0.3) .. (0,0) .. controls (3.31,0.3) and (6.95,1.4) .. (10.93,3.29)   ;
			\draw [color={rgb, 255:red, 0; green, 0; blue, 255 }  ,draw opacity=1 ]   (210.57,129.43) .. controls (221.14,169.43) and (261.21,169.68) .. (269.91,169.69) ;
			\draw [color={rgb, 255:red, 0; green, 0; blue, 255 }  ,draw opacity=1 ]   (269.91,169.69) .. controls (281.71,169.71) and (311.5,169.75) .. (330.29,130.29) ;
			
			\draw (346.75,160.63) node [anchor=north west][inner sep=0.75pt]    {$x$};
			\draw (206.24,89.18) node [anchor=north west][inner sep=0.75pt]    {$y$};
			\draw (246.86,110.68) node [anchor=north west][inner sep=0.75pt]    {$\psi  >0$};
			\draw (248.29,180.23) node [anchor=north west][inner sep=0.75pt]    {$\psi =0$};
			\draw (188.36,152.59) node [anchor=north west][inner sep=0.75pt]    {$O$};
			\draw (260.14,149.83) node [anchor=north west][inner sep=0.75pt]    {$X_{0}$};		
		\end{tikzpicture}
		\caption{Horizontal flat asymptotics.}
		\label{Fig: horizontal flatness asymptotics}
	\end{figure}
	
	If $\Phi(0+)=0$, then (cf. Figure \ref{Fig: cusp asymptotics}) $\gamma_{1}(t)\neq x_{0}$ in $(-t_{1},t_{1})\setminus\{0\}$, $\gamma_{1}-x_{0}$ does not change its sign at $t=0$, and
	\begin{align*}
		\lim_{t\to0}\frac{\gamma(t)}{\gamma_{1}(t)-x_{0}}=0.
	\end{align*}
	\begin{figure}[ht!]

		\tikzset{every picture/.style={line width=0.75pt}} 
		
		\begin{tikzpicture}[x=0.75pt,y=0.75pt,yscale=-1,xscale=1]
			
			\draw    (319.33,187.67) -- (319.43,81.94) ;
			\draw [shift={(319.43,79.94)}, rotate = 90.05] [color={rgb, 255:red, 0; green, 0; blue, 0 }  ][line width=0.75]    (10.93,-3.29) .. controls (6.95,-1.4) and (3.31,-0.3) .. (0,0) .. controls (3.31,0.3) and (6.95,1.4) .. (10.93,3.29)   ;
			\draw    (301.14,149.27) -- (459.84,149.27) ;
			\draw [shift={(461.84,149.27)}, rotate = 180] [color={rgb, 255:red, 0; green, 0; blue, 0 }  ][line width=0.75]    (10.93,-3.29) .. controls (6.95,-1.4) and (3.31,-0.3) .. (0,0) .. controls (3.31,0.3) and (6.95,1.4) .. (10.93,3.29)   ;
			\draw [color={rgb, 255:red, 0; green, 0; blue, 255 }  ,draw opacity=1 ]   (322,124.8) .. controls (330.4,135.6) and (355.2,149.2) .. (381.49,149.27) ;
			\draw [color={rgb, 255:red, 0; green, 0; blue, 255 }  ,draw opacity=1 ]   (326,117.2) .. controls (331.6,125.6) and (354.5,146.25) .. (381.49,149.27) ;
			
			\draw (461.73,138.51) node [anchor=north west][inner sep=0.75pt]    {$x$};
			\draw (321.89,68.4) node [anchor=north west][inner sep=0.75pt]    {$y$};
			\draw (367.33,108.23) node [anchor=north west][inner sep=0.75pt]    {$\psi =0$};
			\draw (369.6,158.45) node [anchor=north west][inner sep=0.75pt]    {$\psi =0$};
			\draw (304,131.82) node [anchor=north west][inner sep=0.75pt]    {$O$};
			\draw (380.29,128.22) node [anchor=north west][inner sep=0.75pt]    {$X_{0}$};

		\end{tikzpicture}
		\caption{Cusp asymptotics}
		\label{Fig: cusp asymptotics}
	\end{figure}
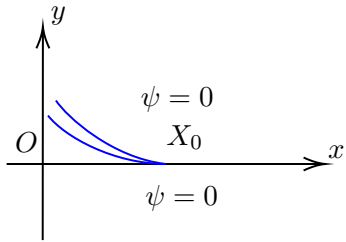
\end{remark}
In the forthcoming three sections, we will focus our analysis on those stagnation points at which the density takes the value $x_{0}\int_{B_{1}}y^{+}dX$, and the asymptotic profile is smooth (Figure \ref{Fig: horizontal flatness asymptotics}). We collect these points together and call them \emph{degenerate points}. The behavior of solutions in proximity to degenerate points will be investigated via a “frequency formula” (Section \ref{Section: frequency formula}). The analysis derived from the “frequency formula” implies that the order of solutions is indeed higher than $3/2$. Then we are able to apply a blow-up analysis to exclude the horizontal flat singularities (Section \ref{Section: blow-up limits} and Section \ref{Section: conclusions}).
\section{Degenerate points and a frequency formula}\label{Section: frequency formula}
This section aims to investigate qualitative properties of degenerate points and introduces a frequency formula. We begin with the definition of the set of degenerate points.
\begin{definition}[Degenerate points]
	Let $\psi$ be a variational solution of \eqref{Formula: model problem1}. We define the set of degenerate points to be the set of free boundary points whose density takes the value $x_{0}\int_{B_{1}}y^{+}dX$. Namely,
	\begin{align*}
		\Sigma_{\psi}:=\left\lbrace X_{0}\in S_{\psi}^{s}\colon\Phi(0+)=x_{0}\int_{B_{1}}y^{+}dX\right\rbrace.
	\end{align*}
\end{definition}
\begin{remark}\label{Remark: uppersemicontinuity}
	The set $\Sigma_{\psi}$ is closed, as a consequence of the upper semicontinuity of the monotonicity formula $X\mapsto\Phi_{\psi,X}(0+)$ in $\{y=0\}$ and the characterization of the set of values of $\Phi(0+)$. (cf. \cite[Proposition 3.3]{DHP2022}).
\end{remark}
We now introduce our strategy of formulating the frequency formula. Consider first a function $\tilde{\Phi}(r)$ defined by
\begin{align}\label{Formula: perturbation of Weiss boundary adjusted energy-1}
	\tilde{\Phi}(r):=\Phi(r)-\int_{0}^{r}t^{-4}\sum_{i=1}^{2}I_{i}(t)dt-\frac{3}{2}\int_{0}^{r}t^{-5}J_{1}(t)dt-\int_{0}^{r}t^{-4}K(t)dt,
\end{align}
where $\Phi(r)$, $I_{i}(r)$, $J_{1}(r)$ and $K(r)$ are defined in \eqref{Formula: Weiss-boundary adjusted energy}, \eqref{Formula: I1(r)}, \eqref{Formula: I2(r)}, \eqref{Formula: J1(r)} and \eqref{Formula: K(r)}, respectively. It follows from \cite[Proposition 3.3]{DHP2022} that the functions $r\mapsto r^{-4}\sum_{i=1}^{2}I_{i}(r)$, $r\mapsto r^{-5}J_{1}(r)$ and $r\mapsto r^{-4}K(r)$ are all integrable functions for all $r\in(0,\delta)$ when assuming the growth assumption \eqref{Formula: growth assumption}. We deduce that $r\mapsto\tilde{\Phi}(r)$ is differentiable for a.e. $r\in(0,\delta)$ and it follows from \eqref{Formula: derivatives of Weiss-boundary adjusted energy} that
\begin{align*}
	\tilde{\Phi}'(r)\geqslant0\quad\text{ for a.e. }r\in(0,\delta).
\end{align*}
Thus $\tilde{\Phi}(r)$ is a nondecreasing function with respect to $r$ and this implies that $\tilde{\Phi}(r)\geqslant\tilde{\Phi}(0+)$. Moreover, the integrability of functions $r\mapsto r^{-4}\sum_{i=1}^{2}I_{i}(r)$, $r\mapsto r^{-5}J_{1}(r)$ and $r\mapsto r^{-4}K(r)$ gives that $\tilde{\Phi}(0+)=\Phi(0+)=x_{0}\int_{B_{1}}y^{+}dX$. Then the inequality $\tilde{\Phi}(r)\geqslant r^{-3}\int_{B_{r}(X_{0})}x_{0}y^{+}dX$ can be rearranged into
\begin{align}\label{Formula: perturbation of Weiss boundary adjusted energy}
	\begin{split}
		&r^{-3}\int_{B_{r}(X_{0})}\left(\frac{1}{x}|\nabla\psi|^{2}-x\psi f(\psi)\right)dX-\frac{3}{2}r^{-4}\int_{\partial B_{r}(X_{0})}\frac{1}{x}\psi^{2}d\mathcal{H}^{1}\\
		&\geqslant r^{-3}\int_{B_{r}(X_{0})}x_{0}y^{+}(1-\chi_{\{\psi>0\}})dX-r^{-3}\int_{B_{r}(X_{0})}(x-x_{0})y^{+}\chi_{\{\psi>0\}}dX+\int_{0}^{r}t^{-4}I_{2}(t)dt\\
		&\quad+\int_{0}^{r}t^{-4}I_{1}(t)dt+\frac{3}{2}\int_{0}^{r}t^{-5}J_{1}(t)dt+\int_{0}^{r}t^{-4}K(t)dt.
	\end{split}
\end{align}
Introduce a new function
\begin{align}\label{Formula: perturbation of V(r)}
	e(r):=r^{4}\int_{0}^{r}t^{-4}I_{2}(t)dt-r\int_{B_{r}(X_{0})}(x-x_{0})y^{+}\chi_{\{\psi>0\}}dX+r^{4}\int_{0}^{r}t^{-4}I_{1}(t)dt+\frac{3}{2}r^{4}\int_{0}^{r}t^{-5}J_{1}(t)dt.
\end{align}
Then \eqref{Formula: perturbation of Weiss boundary adjusted energy} can be rewritten as
\begin{align*}
	&r^{-3}\int_{B_{r}(X_{0})}\left(\frac{1}{x}|\nabla\psi|^{2}-x\psi f(\psi)\right)dX-\frac{3}{2}r^{-4}\int_{\partial B_{r}(X_{0})}\frac{1}{x}\psi^{2}d\mathcal{H}^{1}\\
	&\geqslant r^{-3}\int_{B_{r}(X_{0})}x_{0}y^{+}(1-\chi_{\{\psi>0\}})dX+r^{-4}e(r)+\int_{0}^{r}t^{-4}K(t)dt.
\end{align*}
Dividing both sides by $r^{-4}\int_{\partial B_{r}(X_{0})}\frac{1}{x}\psi^{2}d\mathcal{H}^{1}$, we have
\begin{align}\label{Formula: Mean frequency}
	\begin{split}
		&\frac{r\int_{B_{r}(X_{0})}\left(\frac{1}{x}|\nabla\psi|^{2}-x\psi f(\psi)\right)dX}{\int_{\partial B_{r}(X_{0})}\frac{1}{x}\psi^{2}d\mathcal{H}^{1}}-\frac{3}{2}\geqslant\frac{r\int_{B_{r}(X_{0})}x_{0}y^{+}(1-\chi_{\{\psi>0\}})dX+e(r)}{\int_{\partial B_{r}(X_{0})}\frac{1}{x}\psi^{2}d\mathcal{H}^{1}}+\frac{r^{4}\int_{0}^{r}t^{-4}K(t)dt}{\int_{\partial B_{r}(X_{0})}\frac{1}{x}\psi^{2}d\mathcal{H}^{1}}.
	\end{split}
\end{align}
Roughly speaking, the inequality \eqref{Formula: Mean frequency} can be understood as the difference of the mean frequency and $\frac{3}{2}$. Moreover, it implies the structure of the frequency formula and we now exhibit our first main result.
\begin{theorem}[Frequency formula]\label{Theorem: freq}
	Let $\psi$ be a variational solution of the problem \eqref{Formula: model problem1}. Let $X_{0}\in S_{\psi}^{s}$ and assume that $\psi$ satisfies the growth assumption \eqref{Formula: growth assumption}. Define, for a.e. $r\in(0,\delta)$,
	\begin{align}\label{Formula: D(r)}
		D_{X_{0},\psi}(r)=D(r)=\frac{\displaystyle r\int_{B_{r}(X_{0})}\left(\frac{1}{x}|\nabla\psi|^{2}-x\psi f(\psi)\right)dX}{\displaystyle\int_{\partial B_{r}(X_{0})}\frac{1}{x}\psi^{2}d\mathcal{H}^{1}},
	\end{align}
	and
	\begin{align}\label{Formula: V(r)}
		V_{X_{0},\psi}(r)=V(r)=\frac{\displaystyle r\int_{B_{r}(X_{0})}x_{0}y^{+}(1-\chi_{\{\psi>0\}})dX+e(r)}{\displaystyle\int_{\partial B_{r}(X_{0})}\frac{1}{x}\psi^{2}d\mathcal{H}^{1}},
	\end{align}
	where $e(r)$ is defined in \eqref{Formula: perturbation of V(r)}. Moreover, set 
	\begin{align}\label{Formula: Z(r)}
		Z_{X_{0},\psi}(r)=Z(r)=\frac{\displaystyle\int_{\partial B_{r}(X_{0})}\frac{x-x_{0}}{x^{2}}\psi^{2}d\mathcal{H}^{1}}{\displaystyle\int_{\partial B_{r}(X_{0})}\frac{1}{x}\psi^{2}d\mathcal{H}^{1}}.
	\end{align}
	Then the “frequency”
	\begin{align}\label{Formula: H(r)}
		H_{X_{0},\psi}(r)=H(r)=D(r)-V(r)
	\end{align}
	satisfies for a.e. $r\in(0,\delta)$ the identities
	\begin{align}\label{Formula: H'(r)}
		\begin{split}
			H'(r)&=\frac{2}{r}\int\limits_{\partial B_{r}(X_{0})}\frac{1}{x}\left[\frac{r(\nabla\psi\cdot\nu)}{(\int_{\partial B_{r}(X_{0})}\frac{1}{x}\psi^{2}d\mathcal{H}^{1})^{1/2}}-D(r)\frac{\psi}{(\int_{\partial B_{r}(X_{0})}\frac{1}{x}\psi^{2}d\mathcal{H}^{1})^{1/2}}\right]^{2}d\mathcal{H}^{1}\\
			&\quad+\frac{2}{r}V^{2}(r)+\frac{2}{r}V(r)\left(H(r)-\frac{3}{2}\right)+\frac{1}{r}Z(r)\left(H(r)-\frac{3}{2}\right)+\frac{K(r)}{\int_{\partial B_{r}(X_{0})}\frac{1}{x}\psi^{2}d\mathcal{H}^{1}}
		\end{split}\\
		\begin{split}
			&=\frac{2}{r}\int\limits_{\partial B_{r}(X_{0})}\frac{1}{x}\left[\frac{r(\nabla\psi\cdot\nu)}{(\int_{\partial B_{r}(X_{0})}\frac{1}{x}\psi^{2}d\mathcal{H}^{1})^{1/2}}-H(r)\frac{\psi}{(\int_{\partial B_{r}(X_{0})}\frac{1}{x}\psi^{2}d\mathcal{H}^{1})^{1/2}}\right]^{2}d\mathcal{H}^{1}\\
			&\quad+\frac{2}{r}V(r)\left(H(r)-\frac{3}{2}\right)+\frac{1}{r}Z(r)\left(H(r)-\frac{3}{2}\right)+\frac{K(r)}{\int_{\partial B_{r}(X_{0})}\frac{1}{x}\psi^{2}d\mathcal{H}^{1}}.\label{Formula: H'(r)(1)}
		\end{split}
	\end{align}
	Here $K(r)$ is given in \eqref{Formula: K(r)}.
\end{theorem}
\begin{proof}
	It follows from our definition of $D(r)$, $V(r)$ and \eqref{Formula: J(r)} that
	\begin{align}\label{Formula: tildeI(r)}
		\begin{split}
			&J(r)(D(r)-V(r))\\
			&=r^{-3}\int_{B_{r}(X_{0})}\left(\frac{1}{x}|\nabla\psi|^{2}-x\psi f(\psi)\right)dX-r^{-3}\int_{B_{r}(X_{0})}x_{0}y^{+}(1-\chi_{\{\psi>0\}})dX-r^{-4}e(r)\\
			&=r^{-3}\int_{B_{r}(X_{0})}\left(\frac{1}{x}|\nabla\psi|^{2}-x\psi f(\psi)+xy^{+}\chi_{\{\psi>0\}}\right)dX-\int_{0}^{r}t^{-4}\sum_{i=1}^{2}I_{i}(t)dt-\frac{3}{2}\int_{0}^{r}t^{-5}J_{1}(t)dt\\
			&\quad-r^{-3}\int_{B_{r}(X_{0})}x_{0}y^{+}dX.
		\end{split}
	\end{align}
	Define the quantity
	\begin{align}\label{Formula: perturbation of I(r)}
		\tilde{I}(r):=I(r)-\int_{0}^{r}t^{-4}\sum_{i=1}^{2}I_{i}(t)dt-\frac{3}{2}\int_{0}^{r}t^{-5}J_{1}(t)dt,
	\end{align}
	where $I(r)$, $I_{1}(r)$, $I_{2}(r)$ and $J_{1}(r)$ are defined in \eqref{Formula: I(r)}, \eqref{Formula: I1(r)}, \eqref{Formula: I2(r)} and \eqref{Formula: J1(r)} respectively. Then it follows from \eqref{Formula: I'(r)} that for a.e. $r\in(0,\delta)$,
	\begin{align}\label{Formula: tildeI'(r)}
		\tilde{I}'(r)=r^{-4}\left(2r\int_{\partial B_{r}(X_{0})}\frac{1}{x}(\nabla\psi\cdot\nu)^{2}d\mathcal{H}^{1}-3\int_{\partial B_{r}(X_{0})}\frac{1}{x}\psi\nabla\psi\cdot\nu d\mathcal{H}^{1}+K(r)\right)-\frac{3}{2}r^{-5}J_{1}(r).
	\end{align}
	Observe now that by the definition \eqref{Formula: H(r)} and \eqref{Formula: tildeI(r)},
	\begin{align*}
		H(r)=D(r)-V(r)=\frac{\tilde{I}(r)-r^{-3}\int_{B_{r}(X_{0})}x_{0}y^{+}dX}{J(r)}.
	\end{align*}
	Thus, it follows from \eqref{Formula: J'(r)}, \eqref{Formula: tildeI'(r)} and a direct computation that
	\begin{align}\label{Formula: H'(r)-1}
		\begin{split}
			H'(r)&=\frac{\tilde{I}'(r)}{J(r)}-\frac{(\tilde{I}(r)-r^{-3}\int_{B_{r}(X_{0})}x_{0}y^{+}dX)}{J(r)}\frac{J'(r)}{J(r)}\\
			&=\frac{\displaystyle\left(2r\int_{\partial B_{r}(X_{0})}\frac{1}{x}(\nabla\psi\cdot\nu)^{2}d\mathcal{H}^{1}-3\int_{\partial B_{r}(X_{0})}\frac{1}{x}\psi\nabla\psi\cdot\nu d\mathcal{H}^{1}+K(r)\right)}{\displaystyle\int_{\partial B_{r}(X_{0})}\frac{1}{x}\psi^{2}d\mathcal{H}^{1}}-\frac{3}{2}\frac{Z(r)}{r}\\
			&\quad-(D(r)-V(r))\dfrac{1}{r}\frac{\displaystyle\left(2r\int_{\partial B_{r}(X_{0})}\frac{1}{x}\psi\nabla\psi\cdot\nu d\mathcal{H}^{1}-3\int_{\partial B_{r}(X_{0})}\frac{1}{x}\psi^{2}d\mathcal{H}^{1}-J_{1}(r)\right)}{\displaystyle\int_{\partial B_{r}(X_{0})}\frac{1}{x}\psi^{2}d\mathcal{H}^{1}},
		\end{split}
	\end{align}
	where we have used the fact that $\frac{Z(r)}{r}=\frac{r^{-5}J_{1}(r)}{J(r)}$ (recalling \eqref{Formula: J(r)} and \eqref{Formula: J1(r)}). With the aid of the energy identity \eqref{Formula: Energy identity} mentioned in the Remark \ref{Remark: Energy identity}, it is easy to obtain the identity
	\begin{align}\label{Formula: D(r)(1)}
		D(r)=\frac{r\displaystyle\int_{\partial B_{r}(X_{0})}\frac{1}{x}\psi\nabla\psi\cdot\nu d\mathcal{H}^{1}}{\displaystyle\int_{\partial B_{r}(X_{0})}\frac{1}{x}\psi^{2}d\mathcal{H}^{1}}.
	\end{align}
	Introducing \eqref{Formula: D(r)(1)} into \eqref{Formula: H'(r)(1)} gives
	\begin{align*}
		H'(r)&=\frac{2}{r}\left[\frac{\displaystyle r^{2}\int_{\partial B_{r}(X_{0})}\frac{1}{x}(\nabla\psi\cdot\nu)^{2}d\mathcal{H}^{1}}{\displaystyle\int_{\partial B_{r}(X_{0})}\frac{1}{x}\psi^{2}d\mathcal{H}^{1}}-\frac{3}{2}D(r)\right]-\frac{3}{2}\frac{Z(r)}{r}\\
		&\quad-\frac{2}{r}(D(r)-V(r))\left[D(r)-\frac{3}{2}-\frac{Z(r)}{2}\right]+\frac{K(r)}{\displaystyle\int_{\partial B_{r}(X_{0})}\frac{1}{x}\psi^{2}d\mathcal{H}^{1}}\\
		&=\frac{2}{r}\left[\frac{\displaystyle r^{2}\int_{\partial B_{r}(X_{0})}\frac{1}{x}(\nabla\psi\cdot\nu)^{2}d\mathcal{H}^{1}}{\displaystyle\int_{\partial B_{r}(X_{0})}\frac{1}{x}\psi^{2}d\mathcal{H}^{1}}-D^{2}(r)\right]+\frac{Z(r)}{r}\left[D(r)-V(r)-\frac{3}{2}\right]\\
		&\quad+\frac{2}{r}V(r)\left(D(r)-\frac{3}{2}\right)+\frac{K(r)}{\displaystyle\int_{\partial B_{r}(X_{0})}\frac{1}{x}\psi^{2}d\mathcal{H}^{1}}\\
		&=\frac{2}{r}\left[\frac{\displaystyle r^{2}\int_{\partial B_{r}(X_{0})}\frac{1}{x}(\nabla\psi\cdot\nu)^{2}d\mathcal{H}^{1}}{\displaystyle\int_{\partial B_{r}(X_{0})}\frac{1}{x}\psi^{2}d\mathcal{H}^{1}}-D^{2}(r)\right]\\
		&\quad+\frac{2}{r}V(r)\left(V(r)+H(r)-\frac{3}{2}\right)+\frac{1}{r}Z(r)\left(H(r)-\frac{3}{2}\right)+\frac{K(r)}{\displaystyle\int_{\partial B_{r}(X_{0})}\frac{1}{x}\psi^{2}d\mathcal{H}^{1}},
	\end{align*}
	where we have used the fact that $H(r)=D(r)-V(r)$ in the second identity. Using the identity \eqref{Formula: D(r)(1)} once again, we obtain \eqref{Formula: H'(r)}. Observe now that
	\begin{align*}
		&\int_{\partial B_{r}(X_{0})}\frac{1}{x}[r(\nabla\psi\cdot\nu)-D(r)\psi]^{2}d\mathcal{H}^{1}\\
		&=\int_{\partial B_{r}(X_{0})}\frac{1}{x}[r(\nabla\psi\cdot\nu)-H(r)\psi]^{2}d\mathcal{H}^{1}-2V(r)r\int_{\partial B_{r}(X_{0})}\frac{1}{x}\psi\nabla\psi\cdot\nu d\mathcal{H}^{1}+2H(r)V(r)\int_{\partial B_{r}(X_{0})}\frac{1}{x}\psi^{2}d\mathcal{H}^{1}\\
		&\quad+V^{2}(r)\int_{\partial B_{r}(X_{0})}\frac{1}{x}\psi^{2}d\mathcal{H}^{1}\\
		&=\int_{\partial B_{r}(X_{0})}\frac{1}{x}[r(\nabla\psi\cdot\nu)-H(r)\psi]^{2}d\mathcal{H}^{1}-2V(r)D(r)\int_{\partial B_{r}(X_{0})}\frac{1}{x}\psi^{2}d\mathcal{H}^{1}+2H(r)V(r)\int_{\partial B_{r}(X_{0})}\frac{1}{x}\psi^{2}d\mathcal{H}^{1}\\
		&\quad+V^{2}(r)\int_{\partial B_{r}(X_{0})}\frac{1}{x}\psi^{2}d\mathcal{H}^{1}\\
		&=\int_{\partial B_{r}(X_{0})}\frac{1}{x}[r(\nabla\psi\cdot\nu)-H(r)\psi]^{2}d\mathcal{H}^{1}-V^{2}(r)\int_{\partial B_{r}(X_{0})}\frac{1}{x}\psi^{2}d\mathcal{H}^{1}.
	\end{align*}
	Here we have used \eqref{Formula: D(r)(1)} in the third identity. Then \eqref{Formula: H'(r)(1)} follows.
\end{proof}
\begin{remark}\label{Remark: comparison of two pbm}
	In this remark, we make a comparison between the frequency formula \eqref{Formula: H'(r)}, \eqref{Formula: H'(r)(1)} in our context and that for the two-dimensional problem (cf. \cite[(6.1), (6.2) in Theorem 6.7]{VW2012}), and explain the perturbation idea. Recalling first the frequency formulas for the two-dimensional problem in \cite{VW2012}, which have the form
	\begin{align*}
		H'(r)&=\frac{2}{r}\int_{\partial B_{r}(X_{0})}\left[\frac{r(\nabla\psi\cdot\nu)}{\int_{\partial B_{r}(X_{0})}\psi^{2}d\mathcal{H}^{1}}-D(r)\frac{\psi}{\int_{\partial B_{r}(X_{0})}\psi^{2}d\mathcal{H}^{1}}\right]^{2}d\mathcal{H}^{1}\\
		&\quad+\frac{2}{r}V^{2}(r)+\frac{2}{r}V(r)\left(H(r)-\frac{3}{2}\right)+\frac{K(r)}{\int_{\partial B_{r}(X_{0})}\psi^{2}d\mathcal{H}^{1}}\\
		&=\frac{2}{r}\int_{\partial B_{r}(X_{0})}\left[\frac{r(\nabla\psi\cdot\nu)}{\int_{\partial B_{r}(X_{0})}\psi^{2}d\mathcal{H}^{1}}-H(r)\frac{\psi}{\int_{\partial B_{r}(X_{0})}\psi^{2}d\mathcal{H}^{1}}\right]^{2}d\mathcal{H}^{1}\\
		&\quad+\frac{2}{r}V(r)\left(H(r)-\frac{3}{2}\right)+\frac{K(r)}{\int_{\partial B_{r}(X_{0})}\psi^{2}d\mathcal{H}^{1}}.
	\end{align*}
	It is common that both two formulas consist of a non-negative part, which is the first integral on the right-hand side of $H'(r)$. We also note that in both problems, $D(r)$ is the “ mean frequency” and is different because the governing equations for the two problems are different. Here the major differences arise in $V(r)$, $Z(r)$ and $K(r)$. As for $V(r)$, it is defined to be the non-negative quantity
	\begin{align*}
		V(r):=\frac{r\int_{B_{r}(X_{0})}y^{+}(1-\chi_{ \{ \psi>0\} })dX}{\int_{\partial B_{r}(X_{0})}\psi^{2}d\mathcal{H}^{1}}
	\end{align*}
	in the two-dimensional problem. However, recalling \eqref{Formula: V(r)}, we see that it has no sign and is consist of a non-negative part plus a perturbation term $e(r)$. The components of the perturbation term $e(r)$ actually encode the axisymmetric nature of a problem. In fact, it consists of a term $\int_{B_{r}(X_{0})}(x-x_{0})y^{+}\chi_{\{\psi>0\}}dX$, which can be viewed as the difference of the density with respect to the $x$-variable, and two terms $\int_{0}^{r}\sum_{i=1}^{2}t^{-4}I_{i}(t)+\frac{3}{2}\int_{0}^{r}t^{-5}J_{1}(t)dt$ caused by the axis of symmetry. These terms are vanishing for the two-dimensional problem and we will eventually prove that near each degenerate points, $V(r)\geqslant0$ for every $r\in(0,r_{0})$ with $r_{0}>0$ sufficiently small. We achieve this by showing that the perturbation term $e(r)$ (which has no sign) tends to zero faster than the leading term $r\int_{B_{r}(X_{0})}x_{0}y^{+}(1-\chi_{\{\psi>0\}})dX$ (which is non-negative). On the other hand, note that the term $Z(r)$ is also new in our axisymmetric case and introduces new difficulties. This is based on a simple fact that for axisymmetric problems, the corresponding energy is no longer $\int_{B_{r}(X_{0})}|\nabla u|^{2}dX$, but $\int_{B_{r}(X_{0})}\frac{1}{x}|\nabla u|^{2}dX$, and we have to deal with the derivatives with respect to the weight $\tfrac{1}{x}$ (cf. \eqref{Formula: I'(r)} and \eqref{Formula: J'(r)}). Incidentally, we point out that although the term $K(r)$ in two formulas are also not the same, we still hope to deal with this term similarly as in \cite{VW2012}, based on a simple fact that the weight $\frac{1}{x}$ is uniformly bounded in $B_{r}(X_{0})$ for every $r\in(0,\d)$.
\end{remark}
In what follows, we first consider the term $K(r)$ in the frequency formula associates with the vorticity strength $f$. First note that when the vorticity is negative, the horizontally flat singularities are possible even for the two-dimensional problem. What is of particular interest to us, as suggested by V\v{a}rv\v{a}ruc\v{a} and Weiss in \cite[Page 863]{VW2012}, is the case when the vorticity is $0$ at the free surface, and may have infinitely many sign changes accumulating there. Based on this observation, we impose the following growth assumption on $f$.
\begin{assumption}\label{Assumption: f}
	There exists a constant $C<+\infty$ such that
	\begin{align}\label{Formula: f(z)}
		|f(z)|\leqslant Cz\quad\text{ for all }\quad z\in(0,z_{0}).
	\end{align}
	Note that \eqref{Formula: f(z)} also implies
	\begin{align}\label{Formula: F(z)}
		|F(z)|\leqslant Cz^{2}/2\quad\text{ for all }\quad z\in(0,z_{0}).
	\end{align}
\end{assumption}
The following Lemma was inspired by \cite[(4.11)]{GL1986} and \cite[Lemma 6.9]{VW2012}.
\begin{lemma}\label{Lemma: vort}
	Let $\psi$ be a variational solution of the problem \eqref{Formula: model problem1}. Then for each $X_{0}=(x_{0},y_{0})\in S_{\psi}^{s}$ and for all $r>0$ sufficiently small,
	\begin{align}\label{Formula: vort(1)}
		r\int_{\partial B_{r}(X_{0})}\frac{1}{x}\psi^{2}d\mathcal{H}^{1}=\int_{B_{r}(X_{0})}\left[2\frac{\psi^{2}}{x}-\frac{x-x_{0}}{x^{2}}\psi^{2}+\left(\frac{|\nabla\psi|^{2}}{x}-x\psi f(\psi)\right)(r^{2}-|X-X_{0}|^{2})\right]dX.
	\end{align}
\end{lemma}
\begin{proof}
	An integration by parts yields
	\begin{align}\label{Formula: vort(1-1)}
		\begin{split}
			&\int_{B_{r}(X_{0})}\operatorname{div}\br{\frac{\nabla(\psi^{2})}{x}}\cdot(r^{2}-|X-X_{0}|^{2})dX\\
			&=-\int_{B_{r}(X_{0})}\frac{\nabla(\psi^{2})}{x}\cdot\nabla\br{r^{2}-|X-X_{0}|^{2}}dX+\underbrace{\int_{\partial  B_{r}(X_{0})}\frac{\nabla(\psi^{2})}{x}\cdot\nu\cdot\br{r^{2}-|X-X_{0}|^{2}}d\mathcal{H}^{1}}_{=0}\\
			&=\int_{B_{r}(X_{0})}\psi^{2}\operatorname{div}\br{\frac{\nabla(r^{2}-|X-X_{0}|^{2}}{x}}dX-\int_{\partial  B_{r}(X_{0})}\frac{\psi^{2}}{x}\nabla\br{r^{2}-|X-X_{0}|^{2}}\cdot\frac{X-X_{0}}{|X-X_{0}|}d\mathcal{H}^{1}\\
			&=\int_{B_{r}(X_{0})}\psi^{2}\frac{\Delta\br{r^{2}-|X-X_{0}|^{2}}}{x}dX-\int_{B_{r}(X_{0})}\frac{\psi^{2}}{x^{2}}\pd{(r^{2}-|X-X_{0}|^{2})}{x}dX+2r\int_{\partial  B_{r}(X_{0})}\frac{\psi^{2}}{x}d\mathcal{H}^{1}\\
			&=-4\int_{B_{r}(X_{0})}\frac{\psi^{2}}{x}dX+2\int_{B_{r}(X_{0})}\frac{x-x_{0}}{x^{2}}\psi^{2}dX+2r\int_{\partial B_{r}(X_{0})}\frac{\psi^{2}}{x}d\mathcal{H}^{1}.
		\end{split}
	\end{align}
	On the other hand, noticing that
	\begin{align}\label{Formula: vort(1-2)}
		\operatorname{div}\br{\frac{\nabla(\psi^{2})}{x}}=2\operatorname{div}\br{\psi\frac{\nabla\psi}{x}}=2\nabla\psi\cdot\frac{\nabla\psi}{x}+2\psi\operatorname{div}\br{\frac{\nabla\psi}{x}}=2\frac{|\nabla\psi|^{2}}{x}-2x\psi f(\psi).
	\end{align}
	A combination of \eqref{Formula: vort(1-1)} and \eqref{Formula: vort(1-2)} gives the desired result.
\end{proof}
As a direct corollary of Lemma \ref{Lemma: vort} and the Assumption \ref{Assumption: f}, we obtain the following estimates for $K(r)$.
\begin{corollary}\label{Corollary: vort}
	Let $\psi$ be a variational solution of the problem \eqref{Formula: model problem1} with the growth assumption \eqref{Formula: growth assumption}. Assume that the nonlinearity $f$ satisfies \eqref{Formula: f(z)}. Then there exists $r_{0}>0$ sufficiently small such that
	\begin{align}\label{Formula: vort(2)}
		r\int_{\partial B_{r}(X_{0})}\frac{1}{x}\psi^{2}d\mathcal{H}^{1}\geqslant\int_{B_{r}(X_{0})}\frac{1}{x}\psi^{2}dX\quad\text{ for all }\quad r\in(0,r_{0}),
	\end{align}
	and
	\begin{align}\label{Formula: vort(3)}
		|K(r)|\leqslant C_{0}r\int_{\partial B_{r}(X_{0})}\frac{1}{x}\psi^{2}d\mathcal{H}^{1}\quad\text{ for all }\quad r\in(0,r_{0}).
	\end{align}
	Here $C_{0}$ depends only on $x_{0}$.
\end{corollary}
\begin{proof}
	It follows from \eqref{Formula: vort(1)} and the Assumption \ref{Assumption: f} that
	\begin{align*}
		r\int_{\partial B_{r}(X_{0})}\frac{1}{x}\psi^{2}d\mathcal{H}^{1}\geqslant\int_{B_{r}(X_{0})}\left(2-\frac{x-x_{0}}{x}-Cx^{2}(r^{2}-|X-X_{0}|^{2})\right)\frac{\psi^{2}}{x}dX.
	\end{align*}
	Consequently, there exists a $r_{0}\in(0,r)$ sufficiently small so that
	\begin{align*}
		2-\frac{x-x_{0}}{x}-Cx^{2}(r^{2}-|X-X_{0}|^{2})\geqslant1\quad\text{ for all }\quad r\in(0,r_{0})\quad\text{ and }\quad X\in B_{r}(X_{0}),
	\end{align*}
	and this proves \eqref{Formula: vort(2)}. As for \eqref{Formula: vort(3)}, recalling the definition of $K(r)$ \eqref{Formula: K(r)} and applying \eqref{Formula: F(z)}, one has
	\begin{align*}
		|K(r)|\leqslant\int_{B_{r}(X_{0})}2x_{0}x\cdot\frac{F(\psi)}{x}+6x^{2}\cdot\frac{F(\psi)}{x}dX+r\int_{\partial B_{r}(X_{0})}2x^{2}\cdot\frac{F(\psi)}{x}+x^{2}\cdot\frac{\psi f(\psi)}{x}d\mathcal{H}^{1}.
	\end{align*}
	With the help of \eqref{Formula: vort(2)}, we obtain the desired result.
\end{proof}
\begin{proposition}\label{Proposition: fre(2)}
	Let $\psi$ be a variational solution of \eqref{Formula: model problem1} with the growth assumption \eqref{Formula: growth assumption}. Assume that the nonlinearity $f$ satisfies \eqref{Formula: f(z)}. Let $X_{0}\in\Sigma_{\psi}$, then the following holds, for some $r_{0}\in(0,\delta)$ sufficiently small.
	\begin{enumerate}
		\item There exists a positive constant $C_{1}$ such that
		\begin{align*}
			H(r)-\frac{3}{2}\geqslant-C_{1}r^{2}\quad\text{ for all }\quad r\in(0,r_{0}).
		\end{align*}
		\item The function $r\mapsto e^{\beta r^{2}}J(r)$ is nondecreasing on $(0,r_{0})$.
		\item $r\mapsto\frac{1}{r}V^{2}(r)\in L^{1}(0,r_{0})$.
		\item The function $r\mapsto H(r)$ has a right limit $H(0+)\geqslant\frac{3}{2}$.
		\item The function
		\begin{align*}
			H'(r)-\frac{2}{r}\int\limits_{\partial B_{r}(X_{0})}\frac{1}{x}\left[\frac{r(\nabla\psi\cdot\nu)}{(\int_{\partial B_{r}(X_{0})}\tfrac{1}{x}\psi^{2}\:d\mathcal{H}^{1})^{1/2}}-H(r)\frac{\psi}{(\int_{\partial B_{r}(x^{0})}\tfrac{1}{x}\psi^{2}\:d\mathcal{H}^{1})^{1/2}}\right]^{2}\:d\mathcal{H}^{1}
		\end{align*}
		is bounded below by a function in $L^{1}(0,r_{0})$.
	\end{enumerate}	
\end{proposition}
\begin{proof}
	We deduce from the inequality \eqref{Formula: perturbation of Weiss boundary adjusted energy} that
	\begin{align}\label{Formula: perturbation-1}
		\begin{split}
			&r^{-3}\int_{B_{r}(X_{0})}\left(\frac{1}{x}|\nabla\psi|^{2}-x\psi f(\psi)\right)dX-\frac{3}{2}r^{-4}\int_{\partial B_{r}(X_{0})}\frac{1}{x}\psi^{2}d\mathcal{H}^{1}-\frac{1}{2}r^{-4}\int_{\partial B_{r}(X_{0})}\frac{x-x_{0}}{x^{2}}\psi^{2}d\mathcal{H}^{1}\\
			&\geqslant r^{-3}\int_{B_{r}(X_{0})}x_{0}y^{+}(1-\chi_{\{\psi>0\}})dX+\int_{0}^{r}t^{-4}I_{2}(t)dt-r^{-3}\int_{B_{r}(X_{0})}(x-x_{0})y^{+}\chi_{\{\psi>0\}}dX\\
			&\quad+\int_{0}^{r}t^{-4}I_{1}(t)dt+t^{-5}J_{1}(t)dt+\frac{1}{2}\int_{0}^{r}t^{-5}J_{1}(t)dt-\frac{1}{2}r^{-4}\int_{\partial B_{r}(X_{0})}\frac{x-x_{0}}{x^{2}}\psi^{2}d\mathcal{H}^{1}+\int_{0}^{r}t^{-4}K(t)dt.
		\end{split}
	\end{align}
	Define now 
	\begin{align*}
		U_{1}(r)=\int_{0}^{r}t^{-4}I_{2}(t)dt=\int_{0}^{r}t^{-4}\int_{B_{t}(X_{0})}(x-x_{0})y^{+}\chi_{\{\psi>0\}}dX,
	\end{align*}
	it follows that $U_{1}(r)$ is differentiable and that 
	\begin{align*}
		\int_{0}^{r}t^{-4}I_{2}(t)dt-r^{-3}\int_{B_{r}(X_{0})}(x-x_{0})y^{+}\chi_{\{\psi>0\}}dX=U_{1}(r)-rU_{1}'(r).
	\end{align*}
	Since $U_{1}(0)=0$, we obtain $U_{1}(r)-rU_{1}'(r)=U_{1}(r)-U_{1}(0)-rU_{1}'(r)=r(U_{1}'(\tilde{r})-U_{1}'(r))$, where $\tilde{r}\in(0,r)$. Consider the function $s\mapsto s^{-4}\int_{B_{s}(X_{0})}(x_{0}-x)y^{+}\chi_{\{\psi>0\}}(x)dx$, then for any $\sigma_{1}$, $\sigma_{2}$ and $X_{0}\in S_{\psi}^{s}$, we have
	\begin{align*}
		|U_{1}'(\sigma_{1})-U_{1}'(\sigma_{2})|&=\left|\sigma_{1}^{-4}\int_{B_{\sigma_{1}}(X_{0})}(x_{0}-x)y^{+}\chi_{\{\psi>0\}}dX-\sigma_{2}^{-4}\int_{B_{\sigma_{2}}(X_{0})}(x_{0}-x)y^{+}\chi_{\{\psi>0\}}dX\right|\\
		&=\left|\int_{B_{1}\cap\{\psi(X_{0}+\sigma_{1}X)>0\}}xy^{+}dX-\int_{B_{1}\cap\{\psi(X_{0}+\sigma_{2}X)>0\}}xy^{+}dX\right|\\
		&\leqslant C|\{\psi(X_{0}+\sigma_{1}X)>0\}\setminus\{\psi(X_{0}+\sigma_{2}X)>0\}|.
	\end{align*} 
	We infer from the above calculation that if $|\sigma_{1}-\sigma_{2}|\to0+$, then $|U_{1}'(\sigma_{1})-U_{2}'(\sigma_{2})|\to0+$. Therefore, if $r\in(0,r_{0})$ with $r_{0}$ sufficiently small, one has $|r-\tilde{r}|$ is sufficiently small and this implies that 
	\begin{align}\label{Formula: r2}
		\int_{0}^{r}t^{-4}I_{2}(t)dt-r^{-3}\int_{B_{r}(X_{0})}(x-x_{0})y^{+}\chi_{\{\psi>0\}}dX=0\quad\text{ for all }r\in(0,r_{0}).
	\end{align}
	Similarly, define
	\begin{align*}
		U_{2}(r):=\int_{0}^{r}t^{-5}J_{1}(t)dt=\int_{0}^{r}t^{-5}\int_{\partial B_{t}(X_{0})}\frac{x-x_{0}}{x^{2}}\psi^{2}d\mathcal{H}^{1},
	\end{align*}
	we have for any $\sigma_{1}$, $\sigma_{2}$ and any $X_{0}\in S_{\psi}^{s}$ that
	\begin{align}\label{Formula: U2(s)}
		\begin{split}
			|U_{2}'(\sigma_{1})-U_{2}'(\sigma_{2})|&=\left|\sigma_{1}^{-5}\int_{\partial B_{\sigma_{1}}(X_{0})}\frac{x-x_{0}}{x^{2}}\psi^{2}d\mathcal{H}^{1}-\sigma_{2}^{-5}\int_{\partial B_{\sigma_{2}}(X_{0})}\frac{x-x_{0}}{x^{2}}\psi^{2}d\mathcal{H}^{1}\right|\\
			&=\left|\int_{\partial B_{1}}\frac{1}{(x_{0}+\sigma_{1}x)^{2}}\psi_{\sigma_{1}}^{2}d\mathcal{H}^{1}-\int_{\partial B_{1}}\frac{1}{(x_{0}+\sigma_{2}x)^{2}}\psi_{\sigma_{2}}^{2}d\mathcal{H}^{1}\right|\\
			&\leqslant\left|\int_{\partial B_{1}}\left[\frac{2x_{0}x(\sigma_{2}-\sigma_{1})+(\sigma_{2}^{2}-\sigma_{1}^{2})}{(x_{0}+\sigma_{1}\xi)^{2}(x_{0}+\sigma_{2}\xi)^{2}}\right]\psi_{\sigma_{1}}^{2}d\mathcal{H}^{1}\right|+\left|\int_{\partial B_{1}}\frac{1}{(x_{0}+\sigma_{2}\xi)^{2}}(\psi_{\sigma_{1}}^{2}-\psi_{\sigma_{2}}^{2})d\mathcal{H}^{1}\right|,
		\end{split}
	\end{align}
	where $\psi_{\sigma}(\xi):=\frac{\psi(X_{0}+\sigma\xi)}{\sigma^{3/2}}$ for $\xi\in\partial B_{1}$. A direct computation gives
	\begin{align}\label{Formula: U2(1)}
		\begin{split}
			|\psi_{\sigma_{1}}-\psi_{\sigma_{2}}(x)|&=\left|\int_{\sigma_{1}}^{\sigma_{2}}\frac{d}{d\sigma}\left(\frac{\psi(X_{0}+\sigma\xi)}{\sigma^{3/2}}\right)d\sigma\right|\\
			&=\left|\int_{\sigma_{1}}^{\sigma_{2}}\left(\frac{\nabla\psi(X_{0}+\sigma\xi)\cdot\xi}{\sigma^{3/2}}-\frac{3}{2}\frac{\psi(X_{0}+\sigma\xi)}{\sigma^{5/2}}\right)d\sigma\right|\\
			&=\left|\int_{\sigma_{1}}^{\sigma_{2}}\frac{1}{\sigma}\left(\nabla\psi_{\sigma}\cdot\xi-\frac{3}{2}\psi_{\sigma}\right)d\sigma\right|.
		\end{split}
	\end{align}
	Thanks to Schwardz inequality, we infer from \eqref{Formula: U2(1)} that
	\begin{align}\label{Formula: U2(2)}
		\begin{split}
			\int_{\partial B_{1}}|\psi_{\sigma_{1}}-\psi_{\sigma_{2}}|d\mathcal{H}^{1}&\leqslant n^{2}\omega_{n}^{2}\int_{\partial B_{1}}|\psi_{\sigma_{1}}-\psi_{\sigma_{2}}|^{2}d\mathcal{H}^{n-1}\\
			&\leqslant n^{2}\omega_{n}^{2}\int_{\partial B_{1}}\left(\int_{\sigma_{1}}^{\sigma_{2}}\frac{1}{\sigma}\left|\nabla\psi_{\sigma}\cdot\xi-\frac{3}{2}\psi_{\sigma}\right|\right)^{2}d\mathcal{H}^{n-1}\\
			&\leqslant n^{2}\omega_{n}^{2}\left(\int_{\sigma_{1}}^{\sigma_{2}}\sigma^{-2}d\sigma\right)\int_{\sigma_{1}}^{\sigma_{2}}\int_{\partial B_{1}}\left(\nabla\psi_{\sigma}\cdot\xi-\frac{3}{2}\psi_{\sigma}\right)^{2}d\mathcal{H}^{n-1}\\
			&\leqslant C\left|\frac{1}{\sigma_{1}}-\frac{1}{\sigma_{2}}\right|\int_{\partial B_{1}}\left(\nabla\psi_{\sigma}\cdot\xi-\frac{3}{2}\psi_{\sigma}\right)^{2}d\mathcal{H}^{n-1}.
		\end{split}
	\end{align}
	Since $\psi$ is a continuous function, it follows from \eqref{Formula: U2(s)} and \eqref{Formula: U2(2)} that $|U_{2}'(\sigma_{1})-U_{2}'(\sigma_{2})|\to0+$ provided that $|\sigma_{1}-\sigma_{2}|\to0+$. Observe that
	\begin{align*}
		\int_{0}^{r}t^{-5}J_{1}(t)dt-r^{-4}\int_{\partial B_{r}(X_{0})}\frac{x-x_{0}}{x^{2}}\psi^{2}d\mathcal{H}^{1}=U_{2}(r)-rU_{2}'(r)=r(U_{2}'(\tilde{r})-U_{2}'(r)),
	\end{align*}
	where $0<\tilde{r}<r$. It follows that for all $r\in(0,r_{0})$ with $r_{0}$ sufficiently small,
	\begin{align}\label{Formula: r3}
		\int_{0}^{r}t^{-5}J_{1}(t)dt-r^{-4}\int_{\partial B_{r}(X_{0})}\frac{x-x_{0}}{x^{2}}\psi^{2}d\mathcal{H}^{1}=0\quad\text{ for all }r\in(0,r_{0}).
	\end{align}
	Let now $\varepsilon\in(0,t)$ be sufficiently small such that $\varepsilon+t\in(0,r)$. Define $t_{\varepsilon,s}:=\varepsilon+st\in(\varepsilon,\varepsilon+t)$ for some $s\in(0,1)$ and let $\varphi(t):=\int_{B_{t}}\frac{x-x_{0}}{x^{2}}\psi^{2}dX$. Then $\varphi(t)$ is an absolute continuous function on $(0,1)$, and we obtain
	\begin{align*}
		t^{-5}J_{1}(t)&=t^{-5}\int_{\partial B_{t}(X_{0})}\frac{x-x_{0}}{x^{2}}\psi^{2}d\mathcal{H}^{1}\\
		&=t^{-5}\int_{0}^{1}\int_{\partial B_{t}(X_{0})}\frac{x-x_{0}}{x^{2}}\psi^{2}d\mathcal{H}^{1}ds\\
		&=t^{-5}\int_{0}^{1}\int_{\partial B_{t_{\varepsilon,s}}(X_{0})}\frac{x-x_{0}}{x^{2}}\psi^{2}d\mathcal{H}^{1} ds+O(\varepsilon)\\
		&=t^{-6}(\varphi(t+\varepsilon)-\varphi(\varepsilon))+O(\varepsilon),
	\end{align*}
	where we have used the fact that $t^{-1}(\varphi(t+\varepsilon)-\varphi(\varepsilon))=\int_{0}^{1}\varphi'(t_{\varepsilon,s})ds$. Passing to the limit as $\varepsilon\to0+$ and writing $X=X_{0}+t\xi\in\partial B_{t}(X_{0})$ for $\xi\in\partial B_{1}$ yields
	\begin{align*}
		t^{-5}J_{1}(t)&=t^{-6}\int_{B_{t}(X_{0})}\frac{x-x_{0}}{x^{2}}\psi^{2}dX\\
		&=t^{-4}\int_{B_{1}}\frac{t\xi}{(x_{0}+t\xi)^{2}}(\psi(X_{0}+t\xi)-\psi(X_{0}))^{2}dX(\xi)\\
		&=t^{-3}\int_{B_{1}}\frac{\xi}{(x_{0}+t\xi)^{2}}\1\int_{0}^{t}\nabla\psi(X_{0}+s\xi)\cdot\xi ds\2^{2}dX(\xi),
	\end{align*}
	where we have also used the fact that $\psi(X_{0})=0$ in the second identity since $X_{0}\in S_{\psi}^{s}$. On the other hand, notice that
	\begin{align*}
		t^{-4}I_{1}(t)=t^{-3}\int_{B_{1}}-\frac{\xi}{(x_{0}+t\xi)^{2}}\left(\int_{0}^{t}\nabla\psi(X_{0}+s_{0}\xi)\cdot\xi ds\right)^{2}dX(\xi),
	\end{align*}
	for some $s_{0}\in(0,t)$. Since $t\in(0,r)$ and $r\in(0,r_{0})$ is sufficiently small, we obtain
	\begin{align}\label{Formula: r1}
		t^{-4}I_{1}(t)+t^{-5}J_{1}(t)=0\quad\text{ for all }\quad t\in(0,r).
	\end{align}
    It follows from \eqref{Formula: perturbation-1}, \eqref{Formula: r2}, \eqref{Formula: r3}, \eqref{Formula: r1} and $r^{-3}\int_{B_{r}(X_{0})}x_{0}y^{+}(1-\chi_{\{\psi>0\}})dX\geqslant0$ that for all $r\in(0,r_{0})$ with $r_{0}$ sufficiently small,
    \begin{align}\label{Formula: perturbation-2}
    	\begin{split}
    		&r^{-3}\int_{B_{r}(X_{0})}\left(\frac{1}{x}|\nabla\psi|^{2}-x\psi f(\psi)\right)dX-\frac{3}{2}r^{-4}\int_{\partial B_{r}(X_{0})}\frac{1}{x}\psi^{2}d\mathcal{H}^{1}-\frac{1}{2}r^{-4}\int_{\partial B_{r}(X_{0})}\frac{x-x_{0}}{x^{2}}\psi^{2}d\mathcal{H}^{1}\\
    		&\geqslant\int_{0}^{r}t^{-4}K(t)dt\geqslant-C_{0}\int_{0}^{r}t^{-3}\int_{\partial B_{t}(X_{0})}\frac{1}{x}\psi^{2}d\mathcal{H}^{1}dt.
    	\end{split}
    \end{align}
    Here we applied \eqref{Formula: vort(3)} in the last inequality.
    
    Recalling $\tilde{I}(r)$ defined in \eqref{Formula: perturbation of I(r)}, then for $r$ sufficiently small, 
    \begin{align}\label{Formula: derivation of Y(r)}
    	\begin{alignedat}{5}
    		&\tilde{I}(r)-\frac{3}{2}J(r)-\int_{B_{1}}x_{0}y^{+}dX\\
    		&=I(r)-\frac{3}{2}J(r)-\int_{0}^{r}\sum_{i=1}^{2}t^{-4}I_{i}(t)dt-\frac{3}{2}\int_{0}^{r}t^{-5}J_{1}(t)dt-\int_{B_{1}}x_{0}y^{+}dX\\
    		&=\F(r)-\F(0+)-\int_{0}^{r}\sum_{i=1}^{2}t^{-4}I_{i}(t)dt-\frac{3}{2}\int_{0}^{r}t^{-5}J_{1}(t)dt\\
    		&=\int_{0}^{r}\left[\F'(t)-\sum_{i=1}^{2}t^{-4}I_{i}(t)-\frac{3}{2}t^{-5}J_{1}(t)\right]dt\\
    		&=\int_{0}^{r}2t^{-3}\int_{\partial B_{t}(X_{0})}\frac{1}{x}\left(\nabla\psi\cdot\nu-\frac{3}{2}\frac{\psi}{r}\right)^{2}d\mathcal{H}^{1}dt+\int_{0}^{r}t^{-4}K(t)dt\\
    		&\geqslant-C_{0}\int_{0}^{r}t^{-3}\int_{\partial B_{t}(X_{0})}\frac{1}{x}\psi^{2}d\mathcal{H}^{1}dt,
    	\end{alignedat}
    \end{align}
    where we have used the definition of $\Phi(r)$ \eqref{Formula: Weiss-boundary adjusted energy} in the second identity, the fact on $\Phi'(r)$ \eqref{Formula: derivatives of Weiss-boundary adjusted energy} in the fourth identity and the estimate \eqref{Formula: vort(3)} in the last inequality. Let $Y(r):(0,r_{0})\to\mathbb{R}$ be defined by
    \begin{align}\label{Formula: Y(r)}
    	Y(r)=\int_{0}^{r}t^{-3}\int_{\partial B_{t}(X_{0})}\frac{1}{x}\psi^{2}d\mathcal{H}^{1}dt.
    \end{align}
    It follows from \eqref{Formula: Energy identity}, \eqref{Formula: J(r)} and \eqref{Formula: J'(r)} that
    \begin{align}\label{Formula: Y(r)(0)}
    	\begin{split}
    		\frac{d}{dr}\left(\frac{Y'(r)}{r}\right)&=J'(r)\\
    		&=\frac{2}{r}\left(r^{-3}\int_{B_{r}(X_{0})}\left(\frac{1}{x}|\nabla\psi|^{2}-x\psi f(\psi)\right)dX-\frac{3}{2}r^{-4}\int_{\partial B_{r}(X_{0})}\frac{1}{x}\psi^{2}d\mathcal{H}^{1}-\frac{1}{2}r^{-4}J_{1}(r)\right),
    	\end{split}
    \end{align}
    Thanks to \eqref{Formula: perturbation-2}, we obtain
    \begin{align}\label{Formula: Y'(r)/r(1)}
    	\frac{d}{dr}\left(\frac{Y'(r)}{r}\right)\geqslant-\alpha\frac{Y(r)}{r}.
    \end{align}
    Here $\alpha<+\infty$ is a positive constant. As an application of the \emph{Bessel type differential inequality} \cite[(6.12)]{VW2012}, we have that the function $r\mapsto Y(r)/r^{1/2}$ is a convex function on $(0,r_{0})$. A similar argument as in the proof of \cite[Theorem 6.12]{VW2012} yields that
    \begin{align*}
    	\frac{3}{2}\frac{Y(r)}{r}\leqslant Y'(r)\quad\text{ for all }\quad r\in(0,r_{0}).
    \end{align*}
    This together with \eqref{Formula: derivation of Y(r)} gives that
    \begin{align}\label{Formula: Y(r)(1)}
    	\tilde{I}(r)-\frac{3}{2}J(r)-\int_{B_{1}}x_{0}y^{+}dX\geqslant-\frac{2}{3}C_{0}r^{-2}\int_{\partial B_{r}(X_{0})}\frac{1}{x}\psi^{2}d\mathcal{H}^{1},
    \end{align}
    which proves (1).
    
    (2). Recalling \eqref{Formula: J'(r)}, \eqref{Formula: Y'(r)/r(1)}, and \eqref{Formula: Y(r)(1)}, we have
    \begin{align*}
    	J'(r)=\frac{d}{dr}\left(\frac{Y'(r)}{r}\right)\geqslant2r^{-1}\cdot\left(-\frac{2}{3}C_{0}r^{-2}\int_{\partial B_{r}(X_{0})}\frac{1}{x}\psi^{2}d\mathcal{H}^{1}\right)\geqslant-2\beta rJ(r),
    \end{align*}
    for some $\beta>0$. This proves (2).
    
    (3). Using \eqref{Formula: vort(3)} and the frequency formula \eqref{Formula: H'(r)} obtained previously, we get that for a.e. $r\in(0,r_{0})$.
    \begin{align*}
    	H'(r)\geqslant\frac{2}{r}V^{2}(r)-2C_{1}r|V(r)|-C_{1}r|Z(r)|-C_{0}r.
    \end{align*}
    Recalling \eqref{Formula: Z(r)}, we have
    \begin{align*}
    	Z(r)=\frac{\displaystyle\int_{\partial B_{r}(X_{0})}\frac{x-x_{0}}{x^{2}}\psi^{2}d\mathcal{H}^{1}}{\displaystyle\int_{\partial B_{r}(X_{0})}\frac{1}{x}\psi^{2}d\mathcal{H}^{1}}=\frac{\displaystyle\int_{\partial B_{r}(X_{0})}\frac{1}{x}\psi^{2}\cdot\left(\frac{x-x_{0}}{x}\right)d\mathcal{H}^{1}}{\displaystyle\int_{\partial B_{r}(X_{0})}\frac{1}{x}\psi^{2}d\mathcal{H}^{1}},
    \end{align*}
    Since $X_{0}\in S_{\psi}^{s}$, we have $\frac{x-x_{0}}{x}=O(r)$ and thus
    \begin{align}\label{Formula: order of Z(r)}
    	|Z(r)|\leqslant C_{2}r\quad\text{ for all }\quad r\in(0,r_{0}).
    \end{align}
    Moreover,
    \begin{align*}
    	2C_{1}r|V(r)|\leqslant\frac{1}{r}V^{2}(r)+C_{1}^{2}r^{3},
    \end{align*}
    we get for a.e. $r\in(0,r_{0})$
    \begin{align}\label{Formula: lower bound for H'(r)}
    	H'(r)\geqslant\frac{1}{r}V^{2}(r)-C_{1}^{2}r^{3}-C_{1}C_{2}r^{2}-C_{0}r.
    \end{align}
    Since $r\mapsto H(r)$ is bounded below as $r\to0+$, $r\mapsto H(r)$ must be bounded below as $r\to0+$, this implies that $r\mapsto\tfrac{1}{r}V^{2}(r)\in L^{1}(0,r_{0})$. 
    
    (4). The existence of the limit $\lim_{r\to0+}H(r)$ follows directly from \eqref{Formula: lower bound for H'(r)} and then $H(0+)\geqslant\frac{3}{2}$.
    
    (5). Consider now \eqref{Formula: H'(r)(1)}, a similar argument as in the proof of part (3) gives
    \begin{align}\label{Formula: H'(r)(2)}
    	\begin{split}
    		&H'(r)-\frac{2}{r}\int\limits_{\partial B_{r}(X_{0})}\frac{1}{x}\left[\frac{r(\nabla\psi\cdot\nu)}{(\int_{\partial B_{r}(X_{0})}\frac{1}{x}\psi^{2}d\mathcal{H}^{1})^{1/2}}-D(r)\frac{\psi}{(\int_{\partial B_{r}(X_{0})}\frac{1}{x}\psi^{2}d\mathcal{H}^{1})^{1/2}}\right]^{2}d\mathcal{H}^{1}\\
    		&\geqslant-2C_{1}r|V(r)|-C_{1}r|Z(r)|-C_{0}r\geqslant-\frac{1}{r}V^{2}(r)-C_{1}^{2}r^{3}-C_{1}C_{2}r^{2}-C_{0}r,
    	\end{split}
    \end{align}
    which, together with (3), gives (5).
\end{proof}
To conclude this section, we prove the following results, which states that $V(r)$ can be viewed as a non-negative part $\tilde{V}(r)$ plus a perturbation term.
\begin{lemma}\label{Lemma: V(r)}
	Let $\psi$ be a variational solution of the problem \eqref{Formula: model problem1} and let $V(r)$ be given as in \eqref{Formula: V(r)}. Then for every $X_{0}\in\Sigma_{\psi}$ there exists some sufficiently small $r_{0}\in(0,\delta)$ so that
	\begin{align}\label{Formula: tildeV(1)}
		V(r)=\widetilde{V}(r)+\frac{1}{2}Z(r)\quad\text{ for all }\quad r\in(0,r_{0}).
	\end{align}
	Here,
	\begin{align}\label{Formula: tildeV(2)}
		\widetilde{V}(r):=\frac{\displaystyle r\int_{B_{r}(X_{0})}x_{0}y^{+}(1-\chi_{\{\psi>0\}})\:dX}{\displaystyle\int_{\partial B_{r}(X_{0})}\frac{1}{x}\psi^{2}\:d\mathcal{H}^{1}}\geqslant0.
	\end{align}
\end{lemma}
\begin{proof}
	It follows from the definition of $V(r)$ \eqref{Formula: V(r)} that
	\begin{align*}
		\left(V(r)-\widetilde{V}(r)\right)\int_{\partial B_{r}(X_{0})}\frac{1}{x}\psi^{2}d\mathcal{H}^{1}=e(r).
	\end{align*}
	It follows from \eqref{Formula: r2}, \eqref{Formula: r3} and \eqref{Formula: r1} that $r^{-4}\left(e(r)-\frac{1}{2}\int_{\partial B_{r}(X_{0})}\frac{x-x_{0}}{x^{2}}\psi^{2}d\mathcal{H}^{1}\right)=0$ for all $r\in(0,r_{0})$ with $r_{0}$ sufficiently small. Thus, we obtain
	\begin{align*}
		V(r)-\widetilde{V}(r)=\frac{1}{2}\frac{\displaystyle\int_{\partial B_{r}(X_{0})}\frac{x-x_{0}}{x^{2}}\psi^{2}d\mathcal{H}^{1}}{\displaystyle\int_{\partial B_{r}(X_{0})}\frac{1}{x}\psi^{2}d\mathcal{H}^{1}}\quad\text{ for all }\quad r\in(0,r_{0}).
	\end{align*}
	Then the desired result follows from the definition of $Z(r)$ (recalling \eqref{Formula: Z(r)}).
\end{proof}
\section{Blow-up limits}\label{Section: blow-up limits}
In this section, we consider the blow-up sequence $\{\phi_{m}\}$ defined by 
\begin{align}\label{Formula: vm}
	\phi_{m}(X):=\frac{\psi(X_{0}+r_{m}X)}{\sqrt{r_{m}^{-1}\int\limits_{\partial B_{r_{m}}(X_{0})}\tfrac{1}{x}\psi^{2}\:d\mathcal{H}^{1}}},
\end{align}
where $r_{m}\to0+$ as $m\to\infty$. With the aid of the frequency formula derived in the previous section, we are able to passing to the limit as $m\rightarrow\8$. One particular difficulty that arises here is that it is not obvious whether limits of $\phi_{m}$ are again variational solutions. To overcome this, we would like to employ a \emph{concentration compactness} result for axisymmetric waves introduced in \cite[Theorem 6.1]{VW2014}, making it possible to pass to the limit in the domain variation formula for $\phi_{m}$ (presented in Proposition \ref{Proposition: strong convergence}). We begin with the following Lemma, which indicates that along a subsequence, $v_{m}$ converges weakly to a homogeneous function of degree $H_{X_{0},\psi}(0+)$.
\begin{lemma}\label{Lemma: formula for phim}
	Let $\psi$ be a variational solution and let $X_{0}\in\Sigma_{\psi}$, and let $\phi_{m}$ be the sequence defined in \eqref{Formula: vm}. Then for every $0<\varrho<\sigma<1$,
	\begin{align}\label{Formula: vm(1)}
		\int\limits_{B_{\sigma}\setminus B_{\varrho}}\frac{1}{x_{0}}|X|^{-5}[\nabla\phi_{m}(\tilde{X})\cdot\tilde{X}-H_{X_{0},\psi}(0+)\phi_{m}]^{2}\:dX\to0\quad\text{ as }m\to\infty.
	\end{align}
\end{lemma}
\begin{proof}
	Recalling the inequality \eqref{Formula: H'(r)(2)} and integrating it from $(r_{m}\varrho,r_{m}\sigma)$ with respect to $r$ for every $0<\varrho<\sigma$ gives
	\begin{align}\label{Formula: vm(2)}
		&\int\limits_{r_{m}\varrho}^{r_{m}\sigma}\frac{2}{r}\frac{\int\limits_{\partial B_{r}(X_{0})}\tfrac{1}{x}[r(\nabla\psi\cdot\nu)-H(r)\psi]^{2}\:d\mathcal{H}^{1}}{\int\limits_{\partial B_{r}(X_{0})}\tfrac{1}{x}\psi^{2}\:d\mathcal{H}^{1}}dr\\
		&\leqslant\int_{r_{m}\varrho}^{r_{m}\sigma}H'(r)\:dr+\int_{r_{m}\varrho}^{r_{m}\sigma}\left(\frac{1}{r}V^{2}(r)+C_{1}^{2}r^{3}+C_{1}C_{2}r^{2}+C_{0}r\right)\:dr\nonumber.
	\end{align}
	Changing variables $\tilde{X}:=\frac{X-X_{0}}{r_{m}}$ and set $\tilde{r}:=\frac{r}{r_{m}}$ for simplicity, we rewrite the left-hand side of the inequality \eqref{Formula: vm(2)} as 
	\begin{align*}
		&\int\limits_{\varrho}^{\sigma}\frac{2}{\tilde{r}r_{m}}\frac{\int\limits_{\partial B_{r/r_{m}}}\tfrac{1}{x_{0}+r_{m}\tilde{x}}[\tilde{r}r_{m}\nabla\psi(X_{0}+r_{m}\tilde{X})\cdot\frac{\tilde{X}}{\tilde{r}}-H(\tilde{r}r_{m})\psi(X_{0}+r_{m}\tilde{X})]^{2}r_{m}\:d\mathcal{H}^{1}}{\int\limits_{\partial B_{r/r_{m}}}\tfrac{1}{x_{0}+r_{m}\tilde{x}}\psi^{2}(X_{0}+r_{m}\tilde{X})r_{m}\:d\mathcal{H}^{1}}r_{m}\:d\tilde{r}\\
		&=\int\limits_{\varrho}^{\sigma}\frac{2}{\tilde{r}}\frac{\int\limits_{\partial B_{\tilde{r}}}\tfrac{1}{x_{0}+r_{m}\tilde{x}}[r_{m}\nabla\psi(X_{0}+r_{m}\tilde{X})\cdot\tilde{X}-H(\tilde{r}r_{m})\psi(X_{0}+r_{m}\tilde{X})]^{2}\:d\mathcal{H}^{1}}{\int\limits_{\partial B_{\tilde{r}}}\tfrac{1}{x_{0}+r_{m}\tilde{x}}\psi^{2}(X_{0}+r_{m}\tilde{X})\:d\mathcal{H}^{1}}\:d\tilde{r}
	\end{align*}
	It follows from \eqref{Formula: vm} that
	\begin{align*}
		r_{m}\nabla\psi(X_{0}+r_{m}\tilde{X})=\nabla\phi_{m}(\tilde{X})\sqrt{r_{m}^{-1}\int_{\partial B_{r_{m}}(X_{0})}\frac{1}{x}\psi^{2}\:d\mathcal{H}^{1}}.
	\end{align*}
	Introducing this into the previous integral gives
	\begin{align*}
		&\int\limits_{\varrho}^{\sigma}\frac{2}{\tilde{r}}\frac{\int\limits_{\partial B_{\tilde{r}}}\tfrac{1}{x_{0}+r_{m}\tilde{x}}\left[\sqrt{r_{m}^{-1}\int\limits_{\partial B_{r_{m}}(X_{0})}\tfrac{1}{x}\psi^{2}\:d\mathcal{H}^{1}}\nabla \phi_{m}(\tilde{X})\cdot\tilde{X}-H(\tilde{r}r_{m})\psi(X_{0}+r_{m}\tilde{X})\right]^{2}\:d\mathcal{H}^{1}}{\int\limits_{\partial B_{\tilde{r}}}\tfrac{1}{x_{0}+r_{m}\tilde{x}}\psi^{2}(X_{0}+r_{m}\tilde{X})\:d\mathcal{H}^{1}}\:d\tilde{r}\\
		&=\int\limits_{\varrho}^{\sigma}\frac{2}{\tilde{r}}\frac{r_{m}^{-1}\int\limits_{\partial B_{r_{m}}(X_{0})}\tfrac{1}{x}\psi^{2}\:d\mathcal{H}^{1}\int\limits_{\partial B_{\tilde{r}}}\tfrac{1}{x_{0}+r_{m}\tilde{x}}\left[\nabla\phi_{m}(\tilde{X})\cdot\tilde{X}-H(r_{m}\tilde{r})\tfrac{\psi(X_{0}+r_{m}\tilde{X})}{\sqrt{r_{m}^{-1}\int\limits_{\partial B_{r_{m}}(X_{0})}\tfrac{1}{x}\psi^{2}\:d\mathcal{H}^{1}}}\right]^{2}\:d\mathcal{H}^{1}}{\int\limits_{\partial B_{\tilde{r}}}\tfrac{1}{x_{0}+r_{m}\tilde{x}}\psi^{2}(X_{0}+r_{m}\tilde{X})\:d\mathcal{H}^{1}}\:d\tilde{r}\\
		&=\int\limits_{\varrho}^{\sigma}\frac{2}{\tilde{r}}\int\limits_{\partial B_{\tilde{r}}}\frac{1}{x_{0}+r_{m}\tilde{x}}[\nabla\phi_{m}(\tilde{X})\cdot\tilde{X}-H(\tilde{r}r_{m})\phi_{m}]^{2}\frac{r_{m}^{-1}\int_{\partial B_{r_{m}}(X_{0})}\frac{1}{x}\psi^{2}d\mathcal{H}^{1}}{\int_{\partial B_{\tilde{r}}}\frac{1}{x_{0}+r_{m}\tilde{x}}\psi^{2}(X_{0}+r_{m}\tilde{X})d\mathcal{H}^{1}}d\mathcal{H}^{1}d\tilde{r}\\
		&=\int\limits_{\varrho}^{\sigma}\frac{2}{\tilde{r}}\int\limits_{\partial B_{\tilde{r}}}\frac{1}{x_{0}+r_{m}\tilde{x}}\frac{[\nabla\phi_{m}(\tilde{X})\cdot\tilde{X}-H(\tilde{r}r_{m})\phi_{m}]^{2}}{\int_{\partial B_{\tilde{r}}}\frac{1}{x_{0}+r_{m}\tilde{x}}\phi_{m}^{2}(\tilde{X})d\mathcal{H}^{1}}d\tilde{r}.
	\end{align*}
	This together with \eqref{Formula: vm(2)} gives
	\begin{align*}
		&\int\limits_{\varrho}^{\sigma}\frac{2}{\tilde{r}}\int\limits_{\partial B_{\tilde{r}}}\frac{1}{x_{0}+r_{m}\tilde{x}}\frac{[\nabla\phi_{m}(\tilde{X})\cdot\tilde{X}-H(\tilde{r}r_{m})\phi_{m}]^{2}}{\int_{\partial B_{\tilde{r}}}\frac{1}{x_{0}+r_{m}x}\phi_{m}^{2}(\tilde{X})d\mathcal{H}^{1}}d\tilde{r}\\
		&\leqslant H(r_{m}\sigma)-H(r_{m}\varrho)+\int_{r_{m}\varrho}^{r_{m}\sigma}\left(\frac{1}{r}V^{2}(r)+C_{1}^{2}r^{3}+C_{1}C_{2}r^{2}+C_{0}r\right)\:dr\to0,
	\end{align*}
	as $m\to\infty$. Now note that for every $\tilde{r}\in(\varrho,\sigma)\subset(0,1)$ and all $m$ as above, 
	\begin{align*}
		\int\limits_{\partial B_{\tilde{r}}}\frac{1}{x_{0}+r_{m}\tilde{x}}\phi_{m}^{2}\:d\mathcal{H}^{1}&=\frac{\int\limits_{\partial B_{\tilde{r}}}\tfrac{1}{x_{0}+r_{m}\tilde{x}}\psi^{2}(X_{0}+r_{m}\tilde{X})\:d\mathcal{H}^{1}}{r_{m}^{-1}\int\limits_{\partial B_{r_{m}}(X_{0})}\tfrac{1}{x}\psi^{2}\:d\mathcal{H}^{1}}\\
		&=\frac{r_{m}^{-1}\int\limits_{\partial B_{\tilde{r}r_{m}}(X_{0})}\tfrac{1}{x}\psi^{2}\:d\mathcal{H}^{1}}{r_{m}^{-1}\int\limits_{\partial B_{r_{m}}(X_{0})}\tfrac{1}{x}\psi^{2}\:d\mathcal{H}^{1}}\\
		&=\frac{(r_{m}\tilde{r})^{4}J(r_{m}\tilde{r})}{r_{m}^{4}J(r_{m})}\leqslant\tilde{r}^{4}e^{\beta r_{m}^{2}(1-r^{2})}\to\tilde{r}^{4},
	\end{align*}
	as $m\to\infty$, where we have used Proposition  \ref{Proposition: fre(2)} (ii) and $r_{m}\tilde{r}\leqslant r_{m}$ in the last inequality. Therefore,
	\begin{align*}
		\lim_{m\to\infty}\int\limits_{\varrho}^{\sigma}\frac{2}{\tilde{r}^{5}}\int\limits_{\partial B_{\tilde{r}}}\frac{1}{x_{0}+r_{m}\tilde{x}}\left[\nabla\phi_{m}(\tilde{X})\cdot\tilde{X}-H(\tilde{r}r_{m})\phi_{m}\right]^{2}\:d\mathcal{H}^{1}d\tilde{r}=0.
	\end{align*}
	Thanks to the identity
	\begin{align*}
		\tilde{r}^{-5}\int\limits_{\partial B_{\tilde{r}}}\frac{1}{x_{0}+r_{m}\tilde{x}}[\nabla\phi_{m}(\tilde{X})\cdot\tilde{X}-H(\tilde{r}r_{m})\phi_{m}]^{2}\:d\mathcal{H}^{1}=\frac{d}{d\tilde{r}}\int_{B_{\tilde{r}}}\frac{|\tilde{X}|^{-5}}{x_{0}+r_{m}\tilde{x}}\left[\nabla\phi_{m}(\tilde{X})\cdot\tilde{X}-H(\tilde{r}r_{m})\phi_{m}\right]^{2}\:d\tilde{X}.
	\end{align*}
	We obtain for every $0<\varrho<\sigma<1$,
	\begin{align*}
		\lim_{m\to\infty}\int_{B_{\sigma}\setminus B_{\varrho}}\frac{1}{x_{0}+r_{m}\tilde{x}}|\tilde{X}|^{-5}[\nabla\phi_{m}(\tilde{X})\cdot\tilde{X}-H(\tilde{r}r_{m})\phi_{m}]^{2}\:d\tilde{X}=0.
	\end{align*}
	As $H(\tilde{r}r_{m})\to H(0+)$ uniformly in $|\tilde{X}|$, we obtain the desired result.
\end{proof}
\begin{remark} 
	Let $\phi_{m}$ be defined as in \eqref{Formula: vm} and let $X_{0}\in\Sigma_{\psi}$. Then a direct calculation yields
	\begin{align*}
		\int_{B_{r_{m}}(X_{0})}\frac{1}{x}|\nabla\psi|^{2}dX&=\int_{B_{1}}\frac{1}{x_{0}+r_{m}\tilde{x}}|\nabla\psi(X_{0}+r_{m}\tilde{X})|^{2}r_{m}^{2}d\tilde{X}\\
		&=r_{m}^{-1}\int_{\partial B_{r_{m}}(X_{0})}\frac{1}{x}\psi^{2}d\mathcal{H}^{1}\int_{B_{1}}\frac{1}{x_{0}+r_{m}\tilde{x}}|\nabla\phi_{m}|^{2}d\tilde{X}.
	\end{align*}
	This implies that
	\begin{align*}
		\int_{B_{1}}\frac{1}{x_{0}+r_{m}\tilde{x}}|\nabla\phi_{m}|^{2}dX=\frac{\displaystyle r_{m}\int_{B_{r_{m}}(X_{0})}\frac{1}{x}|\nabla\psi|^{2}dX}{\displaystyle\int_{\partial B_{r_{m}}(X_{0})}\frac{1}{x}\psi^{2}d\mathcal{H}^{1}}.
	\end{align*}
	It follows from \eqref{Formula: D(r)} that
	\begin{align}\label{Formula: D(rm)}
		\begin{split}
			\left|D(r_{m})-\int_{B_{1}}\frac{1}{x_{0}+r_{m}\tilde{x}}|\nabla\phi_{m}(\tilde{X})|^{2}d\tilde{X}\right|&=\left|\frac{r_{m}\displaystyle\int_{B_{r_{m}}(X_{0})}x\psi f(\psi)dX}{\displaystyle\int_{\partial B_{r_{m}}(X_{0})}\frac{1}{x}\psi^{2}d\mathcal{H}^{1}}\right|\\
			&\leqslant C(X_{0})\frac{\displaystyle r_{m}\int_{B_{r_{m}}(X_{0})}\frac{1}{x}\psi^{2}dX}{\displaystyle\int_{\partial B_{r_{m}}(X_{0})}\frac{1}{x}\psi^{2}d\mathcal{H}^{1}}\\
			&\leqslant C(X_{0})r_{m}^{2},
		\end{split}
	\end{align}
	for every $r$ sufficiently small, where we have used \eqref{Formula: vort(2)} in the last inequality. This implies that $\phi_{m}$ is bounded by $D(r_{m})$ in $W^{1,2}$-norms, and in what follows next, we will prove that $\lim_{m\to\infty}D(r_{m})=H(0+)$ and thus $\{\phi_{m}\}$ is a sequence bounded in $W^{1,2}(B_{1})$.
\end{remark}
We are now able to prove the following result.
\begin{proposition}\label{Proposition: vm}
	Let $\psi$ be a variational solution of the problem \eqref{Formula: model problem1}, and let $X_{0}\in\Sigma_{\psi}$. Then
	\begin{enumerate}
		\item There exist $\lim_{r\to0+}V(r)=0$ and $\lim_{r\to0+}D(r)=H(0+)$.
		\item Let $\phi_{m}$ be defined in \eqref{Formula: vm} for any $r_{m}\to0+$ as $m\to\infty$, then the sequence is bounded in $W^{1,2}(B_{1})$.
		\item For any sequence $r_{m}\to0+$ as $m\to\infty$ such that the sequence $\phi_{m}$ converges in $W^{1,2}(B_{1})$ to a blow-up limit $\phi_{0}$, then the function $\phi_{0}$ is a  homogeneous function of degree $H_{X_{0},\psi}(0+)$ in $B_{1}$, and satisfies
		\begin{align*}
			\phi_{0}\geqslant0\quad\text{ in }\quad B_{1},\qquad\phi_{0}\equiv0\quad\text{ in }\quad B_{1}\cap\{y\leqslant0\},\qquad\int_{\partial B_{1}}\frac{1}{x_{0}}\phi_{0}^{2}\:d\mathcal{H}^{1}=1.
		\end{align*}
	\end{enumerate}
\end{proposition}
\begin{proof}
	(1). Assume for the sake of contradiction this is not the case. Let $s_{m}\to0$ be such that the sequence $V(s_{m})$ is bounded away from $0$. The integrability of $r\mapsto\frac{1}{r}V^{2}\in L^{1}(0,r_{0})$ (Proposition \ref{Proposition: fre(2)} (3)) implies that
	\begin{align*}
		\min_{r\in[s_{m},2s_{m}]}V(r)\to0\quad\text{ as }\quad m\to\infty.
	\end{align*}
	It follows from \eqref{Formula: tildeV(1)} that
	\begin{align*}
		\min_{r\in[s_{m},2s_{m}]}\left(\widetilde{V}(r)+\frac{1}{2}Z(r)\right)\to0\quad\text{ as }\quad m\to\infty.
	\end{align*}
	Furthermore, it follows from \eqref{Formula: order of Z(r)} that
	\begin{align*}
		\min_{r\in[s_{m},2s_{m}]}\widetilde{V}(r)\to0\quad\text{ as }\quad m\to\infty.
	\end{align*}
	Let $t_{m}\in[s_{m},s_{2m}]$ be such that $\widetilde{V}(t_{m})\to0$ as $m\to\infty$. For the choice $r_{m}:=t_{m}$ for every $m$, the sequence $\phi_{m}$ given by \eqref{Formula: vm} satisfies \eqref{Formula: vm(1)}. The fact that $V(r_{m})\to0$ implies that $D(r_{m})$ is bounded. It follows from \eqref{Formula: D(rm)} that  $\phi_{m}$ is bounded in $W^{1,2}(B_{1})$. Let $\phi_{0}$ be any weak limit of $\phi_{m}$ along a subsequence. Note that by the compact embedding $W^{1,2}(B_{1})\hookrightarrow L^{2}(\partial B_{1})$, $\phi_{0}$ has norm $1$ on $L^{2}(\partial B_{1})$, since this is true for all $m$. It follows from \eqref{Formula: vm(1)} that $\phi_{0}$ is a homogeneous function of degree $H(0+)$. Thanks to Proposition \ref{Proposition: fre(2)} (2), one has
	\begin{align}\label{Formula: tildeVsm}
		\widetilde{V}(s_{m})&=\frac{s_{m}^{-3}\int_{B_{s_{m}}(X_{0})}x_{0}y^{+}(1-\chi_{\{\psi>0\}})\:dX}{s_{m}^{-4}\int_{\partial B_{s_{m}}(X_{0})}\tfrac{1}{x}\psi^{2}\:d\mathcal{H}^{1}}\nonumber\\
		&\leqslant\frac{s_{m}^{-3}\int_{B_{r_{m}}(X_{0})}x_{0}y^{+}(1-\chi_{\{\psi>0\}})\:dX}{e^{\beta[(r_{m}^{2}/4)-s_{m}^{2}]}\int_{\partial B_{r_{m}/2}(X_{0})}\tfrac{1}{x}\psi^{2}\:d\mathcal{H}^{1}}\nonumber\\
		&\leqslant\frac{e^{3\beta r_{m}^{2}/4}}{2}\frac{\int_{\partial B_{r_{m}}(X_{0})}\tfrac{1}{x}\psi^{2}\:d\mathcal{H}^{1}}{\int_{\partial B_{r_{m}/2}(X_{0})}\tfrac{1}{x}\psi^{2}\:d\mathcal{H}^{1}}\widetilde{V}(r_{m})\nonumber\\
		&=\frac{e^{3\beta r_{m}^{2}/4}}{2\int_{\partial B_{1/2}}\tfrac{1}{x_{0}+r_{m}x}\phi_{m}^{2}\:d\mathcal{H}^{1}}\widetilde{V}(r_{m}).
	\end{align}
	Since, at least along a subsequence 
	\begin{align*}
		\int_{\partial B_{1/2}}\frac{1}{x_{0}+r_{m}x}\phi_{m}^{2}\:d\mathcal{H}^{1}\to\int_{\partial B_{1/2}}\frac{1}{x_{0}}\phi_{0}^{2}\:d\mathcal{H}^{1}>0,
	\end{align*}
	we have from \eqref{Formula: tildeVsm} that $\widetilde{V}(r_{m})\to0$, which implies that $\widetilde{V}(s_{m})\to0$. This together with $V(s_{m})=\widetilde{V}(s_{m})+\frac{1}{2}Z(s_{m})$ imply that $V(s_{m})\to0$ as $m\to\infty$. This contradicts to the choice of $V(s_{m})$. It follows indeed that $V(r)\to0$ as $r\to0+$, and therefore $D(r)\to H(0+)$.
	
	(2). It follows from \eqref{Formula: D(rm)} that the boundedness of the sequence $\phi_{m}$ in $W^{1,2}(B_{1})$ is equivalent to the boundedness of $D(r_{m})$, which is true by the statement (1).
	
	(3). Let $r_{m}\to0+$ be an arbitrary sequence such that $\phi_{m}$ converges weakly to $\phi_{0}$. The fact that $\phi_{0}$ is a homogeneous function of degree $H(0+)$ follows directly from \eqref{Formula: vm(1)}. Since $\phi_{0}$ belongs to $W^{1,2}(B_{1})$, its homogeneity implies that $\phi_{0}$ is continuous in $B_{1}$. The fact that $\int_{\partial B_{1}}\frac{1}{x_{0}}\phi_{0}^{2}\:d\mathcal{H}^{1}=1$ is a consequence of $\int_{\partial B_{1}}\tfrac{1}{x_{0}+r_{m}x}\phi_{m}^{2}\:d\mathcal{H}^{1}=1$ for all $m$. This concludes the proof.
\end{proof}
\begin{proposition}\label{Proposition: strong convergence}
	Let $\psi$ be a variational solution of the problem \eqref{Formula: model problem1} with the growth assumption \eqref{Formula: growth assumption}. Assume that the nonlinearity $f$ satisfies \eqref{Formula: f(z)}. Let $X_{0}\in\Sigma_{\psi}$ and let $r_{m}\to0+$ be such that the sequence $\phi_{m}$ given by \eqref{Formula: vm} converges weakly to $\phi_{0}$ in $W^{1,2}(B_{1})$. Then $\phi_{m}$ converges to $\phi_{0}$ strongly in $W_{\mathrm{loc}}^{1,2}(B_{1}\setminus\{0\})$, $\phi_{0}$ is continuous in $B_{1}$ and $\Delta \phi_{0}$ is a nonnegative Radon measure satisfying $\phi_{0}\Delta \phi_{0}=0$ in the sense of Radon measures in $B_{1}$.
\end{proposition}
\begin{proof}
	The proof is similar to \cite[Theorem 6.1]{VW2014} and we only exhibit the major differences in the following. Note first that
	\begin{align}\label{Formula: vm(2)(1)}
		\begin{split}
			\operatorname{div}\left(\frac{1}{x_{0}+r_{m}x}\nabla\phi_{m}(x)\right)&=\frac{r_{m}}{\sqrt{r_{m}^{-1}\int_{\partial B_{r_{m}}(X_{0})}\tfrac{1}{x}\psi^{2}\:d\mathcal{H}^{1}}}\operatorname{div}\left(\tfrac{\nabla \psi(X_{0}+r_{m}X)}{x_{0}+r_{m}x}\right)\\
			&=\frac{-r_{m}(x_{0}+r_{m}x)f(\psi(X_{0}+r_{m}X))}{\sqrt{r_{m}^{-1}\int_{\partial B_{r_{m}}(X_{0})}\tfrac{1}{x}\psi^{2}\:d\mathcal{H}^{1}}}\\
			&\geqslant\frac{-Cr_{m}(x_{0}+r_{m}x)\psi(X_{0}+r_{m}X)}{\sqrt{r_{m}^{-1}\int_{\partial B_{r_{m}}(X_{0})}\tfrac{1}{x}\psi^{2}\:d\mathcal{H}^{1}}}\\
			&\geqslant-C_{1}r_{m}\phi_{m}\quad\text{ for }\quad\phi_{m}>0,
		\end{split}
	\end{align}
	in the sense of distributions. We infer from \cite[Theorem 8.17]{GT1983} that
	\begin{align*}
		\operatorname{div}\left(\frac{1}{x_{0}+r_{m}x}\nabla\phi_{m}\right)\geqslant-C_{3}(\sigma)r_{m}\quad\text{ in }\quad B_{\sigma}
	\end{align*}
	in the sense of measures. Here we use the fact that $\phi_{m}$ is bounded in $L^{1}(B_{1})$. Set $\mu_{m}:=\operatorname{div}\left(\tfrac{1}{x_{0}+r_{m}x}\nabla\phi_{m}\right)$ for all $m$, and it follows that for each non-negative $\eta\in C_{0}^{\infty}(B_{1})$ such that $\eta=1$ in $B_{(\sigma+1)/2}$
	\begin{align*}
		\int_{B_{(\sigma+1)/2}}\eta\:d\mu_{m}\leqslant C_{4}\quad\text{ for all }\quad m\in\mathbb{N},
	\end{align*}
	where we have used \eqref{Formula: vm(2)(1)} in the third inequality and the fact that $\phi_{m}$ is bounded in $L^{1}(B_{1})$ in the last inequality. Since $x_{0}>0$, we deduce from \eqref{Formula: vm(2)(1)} that $\Delta\phi_{0}$ is a non-negative Radon measure in $B_{1}$. The continuity of $\phi_{0}$ implies therefore that $\phi_{0}\Delta\phi_{0}$ is well defined as a non-negative Radon measure on $B_{1}$. In order to apply the concentration compactness by Delort in \cite{D1992}, we modify each $\phi_{m}$ to
	\begin{align*}
		\tilde{\phi}_{m}:=(\phi_{m}+C_{3}(\sigma)r_{m}(\tfrac{x_{0}x^{2}}{2}+\tfrac{r_{m}x^{3}}{3}))\star\varphi_{m}\in C^{\infty}(B_{1}),
	\end{align*}
	where $\star$ denotes the convolution and $\varphi_{m}$ is a standard mollifier such that
	\begin{align}\label{Formula: vm(3)}
		\operatorname{div}\left(\frac{1}{x_{0}+r_{m}x}\nabla\tilde{\phi}_{m}\right)\geqslant0\quad\text{ in }\quad B_{(\sigma+1)/2},\qquad\int_{B_{(\sigma+1)/2}}\eta d\tilde{\mu}_{m}\leqslant C<+\infty,\quad\text{ for all }m,
	\end{align}
	and
	\begin{align}\label{Formula: vm(4)}
		\|\phi_{m}-\tilde{\phi}_{m}\|_{W^{1,2}(B_{\sigma})}\to0\quad\text{ as }\quad m\to\infty,
	\end{align}
	where $\tilde{\mu}_{m}:=\operatorname{div}\left(\frac{1}{x_{0}+r_{m}x}\nabla\tilde{\phi}_{m}\right)$. Introduce a new set of coordinates $x'=x_{0}+r_{m}x$ and $y'=r_{m}y$ and consider the velocity vector field associated with $\phi_{m}(x',y')$:
	\begin{align*}
		V^{m}(X,Y,Z):=\left(-\frac{1}{x'}\partial_{y'}\phi_{m}\cos\theta,-\frac{1}{x'}\partial_{y'}\phi_{m}\sin\theta,\frac{1}{x'}\partial_{x'}\phi_{m}\right),
	\end{align*}
	where $(X,Y,Z)=(x'\cos\varphi,x'\sin\varphi,y')$. We use the notation $\nabla'=(\partial_{x'},\partial_{y'})$ and $\operatorname{div}'=\partial_{x'}+\partial_{y'}$. We define the velocity vector field associated with the function $\tilde{\phi}_{m}(x',y')$, \begin{align*}
		\widetilde{V}^{m}(X,Y,Z):=\left(-\frac{1}{x'}\partial_{y'}\tilde{\phi}_{m}\cos\theta,-\frac{1}{x'}\partial_{y'}\tilde{\phi}_{m}\sin\theta,\frac{1}{x'}\partial_{x'}\tilde{\phi}_{m}\right).
	\end{align*}
	We have that $\widetilde{V}^{m}$ is smooth and satisfies $\operatorname{div}'\widetilde{V}^{m}(X,Y,Z)=0$ and
	\begin{align*}
		\operatorname{curl}\widetilde{V}^{m}=\operatorname{div}'\left(\frac{1}{x'}\nabla'\tilde{\phi}_{m}\right)(-\sin\theta,\cos\theta,0)\geqslant0,
	\end{align*}
	where we have used \eqref{Formula: vm(3)} in the last inequality. We consider also
	\begin{align*}
		V^{0}(X,Y,Z):=\left(-\frac{1}{x_{0}}\partial_{y'}\phi_{0}\cos\theta,-\frac{1}{x_{0}}\partial_{y'}\phi_{0}\sin\theta,\frac{1}{x_{0}}\partial_{x'}\phi_{0}\right).
	\end{align*}
	Taking into account \eqref{Formula: vm(3)} we are in a position to proceed as a similar argument as in \cite[Theorem 6.1]{VW2014} to conclude the following convergence 
	\begin{align}\label{Formula: vm(5)}
		\frac{1}{x_{0}+r_{m}x}\partial_{x}\phi_{m}\partial_{y}\phi_{m}\to\frac{1}{x_{0}}\partial_{x}\phi_{0}\partial_{y}\phi_{0}
	\end{align}
	and
	\begin{align}\label{Formula: vm(6)}
		\frac{1}{x_{0}+r_{m}x}((\partial_{x}\phi_{m})^{2}-(\partial_{y}\phi_{m})^{2})\to\frac{1}{x_{0}}((\partial_{x}\phi_{0})^{2}-(\partial_{y}\phi_{0})^{2})
	\end{align}
	in the sense of distributions in $B_{\sigma}$ as $m\to\infty$. Observe now that \eqref{Formula: vm(1)} shows that
	\begin{align}\label{Formula: vm(7)}
		\frac{1}{x_{0}+r_{m}x}\left(\nabla\phi_{m}\cdot X-H(0+)\phi_{m}\right)\to0
	\end{align}
	strongly in $L^{2}(B_{\sigma}\setminus B_{\tau})$ as $m\to\infty$. A combination of \eqref{Formula: vm(5)}- \eqref{Formula: vm(7)} and a similar argument as in the proof of \cite[Theorem 6.1]{VW2014} gives that
	\begin{align*}
		\int_{B_{\sigma}\setminus B_{\tau}}(\partial_{x}\phi_{m})^{2}\eta dX\to\int_{B_{\sigma}\setminus B_{\tau}}(\partial_{x}\phi_{0})^{2}\eta dX
	\end{align*}
	for every $0\leqslant\eta\in C_{0}^{0}(B_{\sigma}\setminus B_{\tau})$. Using \eqref{Formula: vm(5)} once more again yields that $\nabla\phi_{m}\to\nabla\phi_{0}$ strongly in $L_{\mathrm{loc}}^{2}(B_{\sigma}\setminus B_{\tau})$ with $0<\tau<\sigma<1$ arbitrarily. It follows that $\nabla\phi_{m}$ converges to $\nabla\phi_{0}$ strongly in $L_{\mathrm{loc}}^{2}(B_{1}\setminus\{0\})$. As a direct application of the strong convergence,
	\begin{align*}
		\int_{B_{1}}\frac{1}{x_{0}}\nabla(\eta x_{0})\cdot\nabla\phi_{0}dX=0\quad\text{ for all }\quad\eta\in C_{0}^{1}(B_{1}\setminus\{0\}).
	\end{align*}
	Since $\phi_{0}=0$ in $B_{1}\cap\{y\leqslant0\}$, we have that $\phi_{0}\Delta\phi_{0}=0$ in the sense of distributions in $B_{1}$.
\end{proof}
With the aid of the strong convergence of $\phi_{m}$ to $\phi_{0}$, let $r_{m}\to0+$ be an arbitrary sequence so that the sequence $\phi_{m}$ defined in \eqref{Formula: vm} converges weakly in $W^{1,2}(B_{1})$ to a limit $\phi_{0}$. It follows from Proposition \ref{Proposition: vm} (3) that $\phi_{0}\geqslant0$, $\phi_{0}$ is continuous and is a homogeneous function of degree $H(0+)\geqslant3/2$. Moreover, $\phi_{0}\Delta\phi_{0}$ is a non-negative Radon measure satisfying $\phi_{0}\Delta\phi_{0}=0$ in $B_{1}$. The strong convergence of $\phi_{m}$ to $\phi_{0}$ in $W^{1,2}(B_{1}\setminus\{0\})$ and $V(r_{m})\to0$ as $m\to\infty$ imply that
\begin{align*}
	0=\int_{B_{1}}|\nabla\phi_{0}|^{2}\operatorname{div}\xi-2\nabla\phi_{0}D\xi\nabla\phi_{0}dX,
\end{align*}
for each $\xi\in C_{0}^{1}(B_{1}\cap\{y>0\};\mathbb{R}^{2})$. Then a similar argument as in the proof of \cite[Theorem 9.1]{VW2012} gives the following result.
\begin{proposition}\label{Proposition: v0}
	Let $\psi$ be a variational solution of the problem \eqref{Formula: model problem1} with the growth assumption \eqref{Formula: growth assumption}. Assume that the nonlinearity $f$ satisfies \eqref{Formula: f(z)}. Then at each point $X_{0}$ of the set $\Sigma_{\psi}$ there exists an integer $N(X_{0})\geqslant2$ such that
	\begin{equation*}
		H_{X_{0},\psi}(0+)=N(X_{0})
	\end{equation*}
	and
	\begin{align}\label{Formula: freq}
		\frac{\psi(X_{0}+rX)}{\sqrt{r^{-1}\int_{\partial B_{r}(X_{0})}\tfrac{1}{x}\psi^{2}d\mathcal{H}^{1}}}\to\frac{\sqrt{x_{0}}\rho^{N(X_{0})}|\sin(N(X_{0})\min(\max(\theta,0),\pi))|}{\sqrt{\int_{0}^{\pi}\sin^{2}(N(X_{0})\theta)d\theta}}\quad\text{ as }\quad r\to0+,
	\end{align}
	strongly in $W_{\mathrm{loc}}^{1,2}(B_{1}\setminus\{0\})$ and weakly in $W^{1,2}(B_{1})$, where $X=(\rho\cos\theta,\rho\sin\theta)$.
\end{proposition} 
\section{Conclusions}\label{Section: conclusions}
In this section, we collect two main results, the first states that $\Sigma_{\psi}$ is locally finite in $\Omega$, while the second one demonstrates that $\Sigma_{\psi}=\0$. Thus, we proved that the set of horizontally flat singularities is an empty set. 
\begin{theorem}
	Let $\psi$ be a variational solution of the problem \eqref{Formula: model problem1} with the growth assumption \eqref{Formula: growth assumption}. Assume that the nonlinearity $f$ satisfies \eqref{Formula: f(z)}. Then the set $\Sigma_{\psi}$ is a finite set locally in $\Omega$.
\end{theorem}
\begin{proof}
	The proof is similar to \cite[Theorem 9.2]{VW2012}, we sketch the proof here.
	
	Assume for the sake of contradiction that there is a sequence of points $X_{m}\in\Sigma_{\psi}$ converging to $X_{0}\in\Omega$ with $X_{m}\neq X_{0}$ for all $m$. It follows from the upper semicontinuity (recalling Remark \ref{Remark: uppersemicontinuity}) that $X_{0}\in\Sigma_{\psi}$. Choose $r_{m}:=2|X_{m}-X_{0}|$, and assume without loss of generality that the sequence $(X_{m}-X_{0})/r_{m}$ is constant, with value $P\in\{(1/2,0),(-1/2,0)\}$. Let $\phi_{m}$ be the sequence defined in \eqref{Formula: vm}, and consider also the sequence
	\begin{align*}
		\psi_{m}(X):=\frac{\psi(X_{0}+r_{m}X)}{r_{m}^{3/2}}.
	\end{align*}	
	Note that each $\psi_{m}$ is a variational solution of the problem
	\begin{align*}
		\begin{cases}
			\operatorname{div}\left(\frac{1}{x_{0}+r_{m}x}\nabla\psi_{m}\right)=-(x_{0}+r_{m}x)r_{m}^{1/2}\psi_{m}f(r_{m}^{3/2}\psi_{m})\quad&\text{ in }\quad B_{\delta/r_{m}}\cap\{\psi_{m}>0\},\\
			|\nabla\psi_{m}|^{2}=(x_{0}+r_{m}x)^{2}y\quad&\text{ on }\quad B_{\delta/r_{m}}\cap\partial\{\psi_{m}>0\}.
		\end{cases}
	\end{align*}
	It is easy check that $\phi_{m}$ is a scalar multiple of $\psi_{m}$, that is,
	\begin{align}\label{Formula: vm(8)}
		\phi_m(X)=\frac{\psi_m(X)}{\sqrt{\int_{\partial B_1}\tfrac{1}{x_0+r_mx}\psi_m^2d\mathcal{H}^{1}}}.
	\end{align}
	Since $X_{m}\in\Sigma_{\psi}$, it follows that $P\in\Sigma_{\psi_{m}}$. Therefore, Proposition \ref{Proposition: fre(2)} (1) shows that
	\begin{align*}
		D_{P,\psi_{m}}(r)\gg V_{P,\psi_{m}}(r)+\frac{3}{2}-C_{1}r^{2}\quad\text{ for all }r\in(0,r_{0}),
	\end{align*}
	where $r_{0}\in(0,1/2)$ is sufficiently small and $C_{1}>0$ is a positive constant. It then follows from Lemma \ref{Lemma: V(r)} that $V(r)=\widetilde{V}(r)+\frac{1}{2}Z(r)$ for all $r\in(0,r_{0})$ and therefore,
	\begin{align*}
		D_{P,\psi_{m}}(r)&\gg\widetilde{V}_{P,\psi_{m}}(r)+\frac{1}{2}Z_{P,\psi_{m}}(r)+\frac{3}{2}-C_{1}r^{2}\\
		&\gg\frac{3}{2}+\frac{1}{2}Z_{P,\psi_{m}}(r)-C_{1}r^{2},
	\end{align*}
	where we have used the fact that $\widetilde{V}_{P,\psi_{m}}(r)\gg0$. Also observe that
	\begin{align*}
		r\int_{B_{r}(P)}(x_{0}+r_{m}x)r_{m}^{1/2}\psi_{m}f(r_{m}^{3/2}\psi_{m})dX&\leqslant Crr_{m}^{2}\int_{B_{r}(P)}\frac{1}{x_{0}+r_{m}x}\psi_{m}^{2}dX\\
		&\leqslant Cr^{2}r_{m}^{2}\int_{\partial B_{r}(P)}\frac{1}{x_{0}+r_{m}x}\psi_{m}^{2}d\mathcal{H}^{1}
	\end{align*}
	where we have used \eqref{Formula: f(z)} and \eqref{Formula: vort(2)} in the second and the third inequality. Therefore,
	\begin{align*}
		r\int_{B_{r}(P)}\frac{1}{x_{0}+r_{m}x}|\nabla\psi_{m}|^{2}dX\gg\1\frac{3}{2}+\frac{1}{2}Z(r)-C_{1}r^{2}-Cr^{2}r_{m}^{2}\2\int_{\partial B_{r}(P)}\frac{1}{x_{0}+r_{m}x}\psi_{m}^{2}d\mathcal{H}^{1}.
	\end{align*}    
	Recalling \eqref{Formula: vm(8)}, we obtain
	\begin{align*}
		r\int_{B_{r}(P)}\frac{1}{x_{0}+r_{m}x}|\nabla\phi_{m}|^{2}dX\gg\1\frac{3}{2}+\frac{1}{2}Z(r)-C_{1}r^{2}-Cr^{2}r_{m}^{2}\2\int_{\partial B_{r}(P)}\frac{1}{x_{0}+r_{m}x}\phi_{m}^{2}d\mathcal{H}^{1}.
	\end{align*}
	It is a consequence of Proposition \ref{Proposition: strong convergence} that $\phi_{m}$ converges strongly in $W^{1,2}(B_{1/4}(P))$ to $\phi_{0}$ given by \eqref{Formula: freq}, hence
	\begin{align*}
		r\int_{B_{r}(P)}\frac{1}{x_{0}}|\nabla\phi_{0}|^{2}dX\gg\1\frac{3}{2}+\frac{1}{2}Z(r)-C_{1}r^{2}\2\int_{\partial B_{r}(P)}\frac{1}{x_{0}}\phi_{0}^{2}d\mathcal{H}^{1}.
	\end{align*}
	However, this contradicts with the fact that
	\begin{align*}
		\lim_{r\to0+}\frac{r\int_{B_{r}(P)}|\nabla\phi_{0}|^{2}dX}{\int_{\partial B_{r}(P)}\phi_{0}^{2}d\mathcal{H}^{1}}=1,
	\end{align*}
	where we have used the fact that $\lim_{r\to0+}Z(r)=0$.
\end{proof}
The next theorem states that the horizontally flat singularities do not exist.
\begin{theorem}
	Let $\psi$ be a weak solution of the problem \eqref{Formula: model problem1} with the growth assumption \eqref{Formula: growth assumption}. Assume that the nonlinearity $f$ satisfies \eqref{Formula: f(z)} and that the free boundary satisfies the Assumption \ref{Assumption: curve assumption}. Then $\Sigma_{\psi}=\0$.
\end{theorem}
\begin{proof}
	The idea of the proof mainly borrows from \cite[Theorem 10.1]{VW2012}
	
	Suppose towards a contradiction that there exists a point $X_{0}\in\Sigma_{\psi}$. We infer from Proposition \ref{Proposition: v0} that there exists an integer $N(X_{0})\gg2$ such that
	\begin{align}\label{Formula: vr}
		\begin{split}
			\phi_{r}(X)&:=\frac{\psi(X_{0}+rX)}{\sqrt{r^{-1}\int_{\partial B_{r}(X_{0})}\tfrac{1}{x}\psi^{2}d\mathcal{H}^{1}}}\to\frac{\sqrt{x_{0}}\rho^{N(X_{0})}|\sin(N(X_{0})\min(\max(\theta,0),\pi))|}{\sqrt{\int_{0}^{\pi}\sin^{2}(N(X_{0})\theta)d\theta}}\quad\text{ as }\quad r\to0+,
		\end{split}
	\end{align}
	strongly in $W_{\mathrm{loc}}^{1,2}(B_{1}\setminus\{0\})$ and weakly in $W^{1,2}(B_{1})$, where $X=(\rho\cos\theta,\rho\sin\theta)$. On the other hand, it follows from that for any $\tilde{B}\subset B_{1}\cap\{y>0\}$, $\phi_{r}>0$ in $\tilde{B}$ for sufficiently small $r$. Consequently, recalling \eqref{Formula: vm(2)(1)},
	\begin{align*}
		\left|\operatorname{div}\left(\frac{1}{x_{0}+rx}\nabla\phi_{r}\right)\right|\leqslant C_{1}r\phi_{r}\quad\text{ in }\quad\tilde{B}
	\end{align*}
	for sufficiently small $r$. It follows that $\Delta\phi_{0}=0$ in $\tilde{B}$. This yields a contradiction to \eqref{Formula: vr} since $N(X_{0})\geqslant2$. Hence $\Sigma_{\psi}$ is indeed empty.
\end{proof}


\begin{thebibliography}{60}
\bibitem{A2000} F. J. Almgren, \newblock Almgren's Big Regular Paper, World Scientific Monograph Series in Mathematcis, 1. World Scientific, River Edge, NJ, (2000). 
\bibitem{AC1981} H. W. Alt, L. A. Caffarelli, \newblock Existence and regularity for a minimum problem with free boundary, {\it J. Reine Angew. Math.}, {\bf 325}, 105-144, (1981).
\bibitem{D1992} Delort, J.-M., \newblock Une remarque sur le probleme des nappes de tourbillon axisymetriques sur $\mathbb{R}^{3}$, {\it J. Funct. Anal.}, {\bf 108}, (1992).
\bibitem{DHP2022} L. L. Du, J. L. Huang, Y. Pu, \newblock The free boundary of steady axisymmetric inviscid flow with vorticity I: near the degenerate point, {\it Commun. Math. Phys.}, {\bf 400}, 2137-2179, (2023).
\bibitem{DY2023} L. L. Du, C. L. Yang, \newblock The free boundary of steady axisymmetric inviscid flow with vorticity II: near the non-degenerate points, {\it  arXiv:2310.09477v1}.
\bibitem{GL1986} N. Garofalo, F. H. Lin, \newblock Monotonicity properties of variational integrals: $A_{p}$ weights and unique continuation, {\it Indiana Univ. Math. J}, {\bf 35}, no. 2, 245-268, (1986).
\bibitem{GL1987} N. Garofalo, F. H. Lin, \newblock Unique continuation for elliptic operators: a geometric-variational approach, {\it Comm. Pure Appl, Math}, {\bf XL}, 347-366, (1987).
\bibitem{GT1983} D. Gilbarg, N S. Trudinger, \newblock Elliptic Partial Differential Equations of Second Order, Second Edition, {\it Grundlehren der Mathematischen Wissenschaften}, {\bf 224}, Springer-Verlag, Berlin, (1983).
\bibitem{VW2011} V\u{a}rv\u{a}ruc\u{a}, G. Weiss,  \newblock A geometric approach to generalized Stokes conjectures, {\it Acta Math.}, {\bf 206}(2), 363-403, (2011).
\bibitem{VW2012} E. V\u{a}rv\u{a}ruc\u{a}, G. Weiss,   \newblock The stokes conjecture for waves with vorticity, {\it Ann. Inst. H. Poincar\'{e}.}, {\bf 29}, 861-885, (2012).
\bibitem{VW2014} E. V\u{a}rv\u{a}ruc\u{a}, G. Weiss,  \newblock Singularities of steady axisymmetric free surface flows with gravity, {\it Comm. Pure Appl. Math}, {\bf 67}(8), 1263-1306, (2014).
\end{thebibliography}
\end{document}